\documentclass[bibliography=totocnumbered]{article}
\usepackage{geometry}

    \newgeometry{vmargin={38mm}, hmargin={30mm,30mm}}

\usepackage[utf8]{inputenc}
\usepackage{mathtools}
\usepackage{amsmath}
\usepackage{amsthm}
\usepackage{amssymb}
\usepackage{amsfonts}
\usepackage{upgreek}
\usepackage{centernot}
\usepackage{stmaryrd}
\usepackage{comment}
\usepackage{enumerate}
\usepackage[shortlabels]{enumitem}
\usepackage[linguistics]{forest}
\usepackage{float}
\usepackage{wasysym}
\usepackage{caption}
\usepackage{subcaption}
\usepackage{tikz-qtree}
\usepackage{adjustbox}
\usepackage{graphicx}

\usepackage{epsfig}

\usepackage{accents}

\theoremstyle{plain}
\newtheorem{theorem}{Theorem}[section]
\newtheorem{lemma}[theorem]{Lemma} 
\theoremstyle{definition}
\newtheorem{defn}[theorem]{Definition} 
\newtheorem{prop}[theorem]{Proposition}
\newtheorem{eg}[theorem]{Example}
\newtheorem{rmk}[theorem]{Remark} 

\newtheorem{cor}[theorem]{Corollary}

\def\R{\mathbb{R}}

\def\Z{\mathbb{Z}}
\def\N{\mathbb{N}}
\def\Q{\mathbb{Q}}

\def\S{\Sigma}

\DeclareMathOperator{\Smash}{\mathsf{\Lambda}}
\def\cc{\mathcal{C}}

\def\ff{\mathcal{F}}

\def\pp{\mathcal{P}}

\usepackage{tikz-cd}
\tikzset{
  symbol/.style={
    draw=none,
    every to/.append style={
      edge node={node [sloped, allow upside down, auto=false]{$#1$}}}
  }
}
\usepackage{mathdots}
\usepackage{young}

\usepackage{hyperref}
\hypersetup{colorlinks,linkcolor={red},citecolor={blue},urlcolor={red}}

\def\ss{\mathcal{S}}

\title{Filtered Topology and Persistence in Stable Homotopy}

\author{John Miller}
\date{\today}
\begin{document}

\maketitle{}

\tableofcontents

\section{Introduction}

The theory of persistence modules is relatively new, first introduced in 2005 by Afra Zomorodian and Gunnar Carlsson in their work \cite{ZC}, it formalises the notion of persistent homology introduced in \cite{ELZ} by Herbert Edelsbrunner, David Letscher, and Afra Zomorodian, which was constructed as a tool in topological data analysis. Despite the very applied nature of its beginnings, persistence theory has quickly become a useful tool in various more pure mathematical studies, in particular its relation to Morse theory and the study of functionals on topological spaces. Given a functional $f:X \to \R$ (assumed smooth), we obtain a version of persistence homology by setting
\begin{equation*}\text{H}_*^{\leq r}(X,f):=\text{H}_*(f^{-1}\{\leq r\}; \R)\end{equation*}
the singular homology with $\R$ coefficients of the $r$-sublevel set of $X$ with respect to the functional $f$. For all $r \leq s$ there are morphisms 
\begin{equation*}\text{H}_*^{\leq r}(X,f) \xrightarrow[]{\iota_{r,s}}\text{H}_*^{\leq s}(X,f)\end{equation*}
given by the induced morphism in Homology of the inclusion maps $i_{r,s}:f^{-1}\{\leq r\} \hookrightarrow f^{-1}\{\leq s\}$, noting that if $r=s$, then the induced morphism is exactly the identity. Thus we obtain an $\R$-indexed family of vector spaces, along with a collection of $\R$-linear morphisms $\iota_{r,s}$ for all $s\leq r$. This is the basic definition of a persistence module; we will recall in more detail the basic theory of persistence modules. For the purposes of this paper we relax the assumption of working over a field and will consider a persistence module to be an $\R$-indexed collection of $k$-modules, for some commutative ring $k$, along with $k$-linear morphisms $\iota_{r,s}$. In particular, we will work with Abelian groups. 

The theory of persistence can be abstracted to the language of category theory, in \cite{BCZ1} the notion of persistence categories is introduced. A persistence category is a category enriched in the category of persistence modules. See \cite{BCZ1} and also \cite{BS} for a discussion on this category. If a morphism $f\in \text{Hom}(X,Y)(r)$ lies in the index $r$ level of the persistence module $\text{Hom}(X,Y)$, we say it is a morphism of shift $r$. The composition of morphisms acts additively on shift. There are two canonical categories ($\text{Mod}_k$-enriched) associated to a persistence category; the zero level category $\cc_0$, which is the subcategory of $\cc$ consisting of morphisms of shift zero, and the limit category $\cc_\infty$, this is roughly speaking the category where we 'forget' the size of morphisms. It is shown in \cite{BCZ1} that for the case of persistence categories with an additional triangulated structure, i.e., triangulated persistence categories (TPCs), that $\cc_\infty$ is equivalent to the Verdier localisation of $\cc_0$ with respect to 'weighted acyclic' objects. The structure of TPCs allows one to construct various notions of sizes and distances on the class of objects, the motivation for which originates in the study of filtered Fukaya categories associated to some symplectic manifold, though the framework is purely algebraic. More examples of TPCs have been constructed, for example, in \cite{BCZ1}. It is shown that filtered dg-categories and in particular a category of filtered chain complexes also admit a TPC structure. 

 The objective of this paper is to construct a persistence category that contains the data of topological spaces equipped with functionals. In fact, we will consider general filtrations on spaces which consist of an $\R$-indexed collection of subspaces. We construct a category whose objects are pairs $(X,\ff_X)$ consisting of a space $X$ and a filtration $\ff_X$, and denote this by $\ff \text{Top}_*$. In section \ref{filttopsec} we describe in detail how the category is constructed and some of its properties. In doing so, we define notions of filtered products and wedge sums, as well as a filtered smash product. We then show that there exists a filtered version of CW approximation for filtered spaces.

\begin{lemma}[Filtered CW approximation]\label{CWspace}
Given any filtered space $X\in \ff \text{Top}_*$ there exists a filtered CW complex $\bar{X}$ and a map $f: \bar{X} \to X$ of shift zero, such that $f(r): \bar{X}( r) \to X(r)$ is a weak equivalence for all $r$, and moreover the following homotopy commutes for all $r \leq s$
\begin{equation*}
    \begin{tikzcd}
        \bar{X}(s)\ar[r,"f(s)"] & X(s)\\
        \bar{X}(r)\ar[r,"f(r)"]\ar[u,"i_{r,s}^{\bar{X}}"] & X(r).\ar[u,"i^X_{r,s}"]
    \end{tikzcd}
\end{equation*}

\end{lemma}

With this result, we restrict our attention to filtered CW complexes and explore a weighted version of the Euler characteristic $\hat{\chi}_{\text{CW}}$ that takes values in $\Lambda_P=\{\sum_{i=1,\hdots,n}a_i \cdot t^{r_i}: a_i\in \Z, r_i \in \R\}$. We show that it is a `filtered homotopy' invariant and that evaluation of the polynomial at $t=1$ recovers the usual (reduced) Euler characteristic. Moreover, we show that evaluation at $t=1$ of its derivative is a filtered homotopy invariant and gives a weighted version of the usual Euler characteristic. We furthermore show that $\hat{\chi}_{\text{CW}}$ is additive with respect to the filtered wedge sum and distributive over the filtered smash product:

\begin{lemma}
    The weighted Euler polynomial satisfies the following equalities:
    \begin{align*}
     \hat{\chi}_{\text{CW}}(X \vee Y)=&  \hat{\chi}_{\text{CW}}(X) +  \hat{\chi}_{\text{CW}}(Y)\\
     \hat{\chi}_{\text{CW}}(X \bar{\Smash} Y)=&  \hat{\chi}_{\text{CW}}(X) \cdot  \hat{\chi}_{\text{CW}}(Y).
    \end{align*}
    \end{lemma}

At this point we will have constructed a suitable category of filtered topological spaces and discussed some of its filtered homotopical properties. In order to utilise the framework of persistence modules and persistence categories we will need to 'stabilise' this category. We will construct a persistence version of the Spanier-Whitehead category for CW-complexes (see \cite{SW} for the first introduction to this and also \cite{DP}, \cite{EKM} and \cite{De}). We will denote this category by $\pp \text{SW}$, its objects will be pairs, consisting of filtered CW-complexes and integers, with morphisms given by persistence modules 
\begin{equation*}
    \text{Hom}_{\pp \text{SW}}((X,n),(Y,m)):= \text{Colim}_l[\S^{l+n}X,\S^{l+m}Y] 
\end{equation*}

We show the following:

\begin{theorem}
    The persistence Spanier-Whitehead category $\pp \text{SW}$ is a TPC.
\end{theorem}
We then go on to explore what the fragmentation distances defined in \cite{BCZ1} are measuring in this setting, and show that the weighted Euler characteristic can be recovered from this. We then discuss the $K$-group of $\pp \text{SW}_0$ (which for general TPCs is explored in \cite{BCZ2}) and prove the following result:
\begin{theorem}\label{Mainiso}
    There is an isomorphism of $\Lambda_P$-algebras induced by the weighted Euler characteristic
    \begin{equation*}
        \mathcal{X}:K(\pp \text{SW}_0 ) \to \Lambda_P
    \end{equation*}
    given by $\mathcal{X}([(X,n)])=(-1)^n \cdot \hat{\chi}_{\text{CW}}(X)$.
\end{theorem}
This result reflects a known result that $K(\text{SW})\cong \Z$ with isomorphism induced by the usual (reduced) Euler characteristic. See for example \cite{Mo}. Finally, we will briefly discuss how to extend these ideas to the construction of a persistence stable homotopy category of filtered spectra and comment on why this should be of interest.

We begin by reviewing some of the basics of persistence modules and persistence categories.
\subsection{Review of persistence theory}

We recall some of the basic definitions of persistence theory, see \cite{PRSZ} and also \cite{BS} for more detailed accounts. Furthermore we use \cite{BCZ1} for reference when recalling the definitions of persistence categories and triangulated persistence categories. A \textbf{persistence module} will be a functor $V:(\R,\leq) \to \text{Mod}_k$, where $(\R,\leq)$ is the poset category of reals. Explicitly, $\text{Obj}((\R,\leq))=\R$ and 
\begin{equation*}
    \text{Hom}_{(\R,\leq)}(r,s):=\begin{cases}
        i_{r,s} & r\leq s\\
        \emptyset & r>s.
    \end{cases}
\end{equation*}
Given a persistence module $V$, and a real $t$, one can define the $t$-shift of $V$ to be the persistence module $V[t]:(\R,\leq) \to \text{Mod}_k$, given by
\begin{align*}
    V[t](r)=&V(r+t)\\
    V[t](i_{r,s})=&V(i_{r+t,s+t})\notag
\end{align*}

One then defines a \textbf{persistence module morphism} from $V$ to $W$ of \text{shift} $s$, to be a natural transformation $f:V \to W[s]$. Explicitly one realises this as a $\R$-index family of $k$-linear morphisms 
\begin{equation*}
    f(r):V(r) \to W(r+s)
\end{equation*}
which commute with the \text{persistence module structure maps} $i_{r,t}^V:=V(i_{r,t})$. One can compose a persistence module morphism $f:V \to W[s]$ of shift $s$ with a persistence module morphism $g:W \to Q[s']$ of shift $s'$ to obtain a persistence module morphism $g \circ f$ of shift $s+s'$ by setting

\begin{equation*}
    (g \circ f)(r)= g(r+s) \circ f(r) 
\end{equation*}

The collection of persistence modules along with persistence module morphisms forms a category, $\text{Mod}_k^\pp$, with 
\begin{equation*}
    \text{Hom}_{\text{Mod}_k^\pp}(V,W):= \coprod_{r\in \R} \text{Nat}(V,W[r]).
\end{equation*}

One finds $\text{Hom}_{\text{Mod}_k^\pp}(V,W)$ carries a canonical persistence module structure $\text{Hom}_{\text{Mod}_k^\pp}(V,W):(\R,\leq) \to \text{Mod}_k$, given by
\begin{equation*}
    \text{Hom}_{\text{Mod}_k^\pp}(V,W)(r)=\text{Nat}(V,W[r])
\end{equation*}
and with $\text{Hom}_{\text{Mod}_k^\pp}(V,W)(i_{r,s})$ defined by 
\begin{equation*}
    \text{Hom}_{\text{Mod}_k^\pp}(V,W)(i_{r,s})(f)(t)= W(i_{t+r,t+s}) \circ f(t)=f(t+s-r) \circ V(i_{t,t+s-r}). 
\end{equation*}

For more details and discussion of $\text{Mod}_k^\pp$, see \cite{BM}. We will call a category $\cc$ a \text{persistence category} (following \cite{BCZ1}) if $\text{Hom}_\cc(A,B)$ carries the structure of a persistence module and, moreover, the composition respects this structure. In particular, one requires $\circ:\text{Hom}_\cc(A,B) \times \text{Hom}_\cc(B,C) \to \text{Hom}_\cc(A,C)$ to restrict to a family of $k$-linear maps

\begin{equation*}
    \circ_{r,s}:\text{Hom}_\cc(A,B)(r) \oplus \text{Hom}_\cc(B,C)(s) \to \text{Hom}_\cc(A,C)(r+s)
\end{equation*}

satisfying 
\begin{equation*}
 \circ_{r',s'} \circ \big(\text{Hom}_\cc(A,B)(i_{r,r'}) \oplus  \text{Hom}_\cc(B,C)(i_{s,s'}) \big)= 
 \text{Hom}_\cc(A,C)(i_{r+s,r'+s'}) \circ \circ_{r,s}.
\end{equation*}
We refer to morphisms $f\in \text{Hom}_\cc(A,B)(r)$ as morphisms of shift $r$, and to simplify notation we will usually just write $i_{r,s}$ instead of $\text{Hom}_\cc(A,B)(i_{r,s})$.
\begin{rmk}
    The category of persistence modules is naturally a persistence category.
\end{rmk}

Given two persistence categories $\cc$ and $\cc'$ one defines a \textbf{persistence functor} to be a functor $F: \cc \to \cc'$, such that $F:\text{Hom}_\cc(A,B) \to \text{Hom}_{\cc'}(F(A),F(B))$ is a natural transformation of persistence modules with shift zero. Given two persistence functors $F,F':\cc \to \cc'$ one can define a persistence module of \text{persistence natural transformations}. Explicitly,

\begin{equation*}
    \text{Nat}(F,F')(r):=\{(\eta:F \to F') :\eta_A\in \text{Hom}_{\cc'}(F(A),F'(A))(r) \hspace{5pt} \forall A \in \cc \}
\end{equation*}
and the persistence structure morphisms $\text{Nat}(F,F')(i_{r,s})$ are given by

\begin{equation*}
    \text{Nat}(F,F')(i_{r,s})(\eta)_A=\text{Hom}_{\cc'}(F(A),F'(A))(i_{r,s})
\end{equation*}

Given a persistence category $\cc$, one defines a \textbf{shift functor} on $\cc$ to be a functor $\ss:(\R,+) \to \text{End}_\pp(\cc)$, where $(\R,+)$ is the groupoid of reals with $\text{Obj}((\R,+))=\R$ and $\text{Hom}_{(\R,+)}(r,s):=\eta_{r,s}$ for any $r,s\in \R$. Note that $1_r=\eta_{r,r}$. $\text{End}_\pp(\cc)$ is the persistence category of (persistence) endofunctors on $\cc$, with morphisms given by persistence natural transformations. More explicitly, a shift functor is an $\R$-indexed family of functors $\ss^a=\ss(a):\cc \to \cc$ along with natural isomorphisms $\ss(\eta_{a,b})=\eta_{a,b}:\ss^a \to \ss^b$. We will denote by $(\eta_{a,b})_A\in \text{Hom}_\cc(\ss^aA,\ss^bA)(b-a)$ the associated isomorphism on object $A\in \cc$.
Associated to a persistence category $\cc$ are two categories $\cc_0$, the zero level category, and $\cc_\infty$, the limit category. $\cc_0$ is simply the subcategory of $\cc$ consisting of morphisms with shift zero. The limit category is in some sense the category where we forget the size of the morphism shifts. $\cc_\infty$ has the same objects of $\cc$, but with hom-sets given by equivalence classes of morphisms under the equivalence $\sim$. Here a morphism $f\in \text{Hom}_\cc(A,B)(s)$ and a morphism $f\in \text{Hom}_\cc(A,B)(s')$ are equivalent iff there exists some $t\in \R$ such that $i_{s,t}(f) = i_{s',t}(f')$. We will denote 
\begin{equation*}
    \eta_{r}^A:=i_{-r,0}\circ (\eta_{0,-r})_A \in \text{Hom}_\cc(A,\ss^{-r}A)(0)=\text{Hom}_{\cc_0}(A,\ss^{-r}A)
\end{equation*}
where $r\geq 0$. Note that $\eta_0^A=1_A$ and $\eta_s^{\ss^{-r}A} \circ \eta_r^A=\eta_{r+s}^A $. An object $A\in \cc$ will be called \textbf{weighted acyclic} of weight $r$, or simply $r$-acyclic, if $\eta_{r}^A=0$, and we will denote this by $A\simeq_r 0$. A persistence category $\cc$ is said to be a \textbf{triangulated persistence category} or TPC for short, if $\cc_0$ is triangulated in the usual sense, and if every morphism $\eta_{r}^A$ can be completed to an exact triangle in $\cc_0$
\begin{equation*}
    \begin{tikzcd}
        A \ar[r,"\eta_r^A"]& \ss^{-r}A \ar[r]& K\ar[r] & TA
    \end{tikzcd}
\end{equation*}
with $K\simeq_r 0$. Furthermore, one requires that the triangulation functor $T:\cc_0 \to \cc_0$ commute with the shift functors $\ss^r$. 
\begin{prop}\cite{BCZ1}
    If $\cc$ is a TPC, then $\cc_\infty$ is triangulated and is equivalent as a triangulated category to the localisation of $\cc_0$ with respect to the full subcategory of weighted acyclics.
\end{prop}

A morphism in $f\in \text{Hom}_{\cc_0}(A,B)$ is called a \textbf{weighted isomorphism} of weight $r$, or just an $r$-isomorphism if it can be completed to an exact triangle 
\begin{equation*}
    \begin{tikzcd}
        A\ar[r,"f"] & B\ar[r,"g"] & C \ar[r,"h"]& TA
    \end{tikzcd}
\end{equation*}
with $C\simeq_r 0$. By definition an $r$-isomorphism will represent an isomorphism in $\cc_\infty$.
There is a class of triangles in $\cc_0$ called \textbf{strict exact triangles}, which contains the class of exact triangles. A triangle

\begin{equation*}\begin{tikzcd}
    A\ar[r,"u"] & B\ar[r,"v"] & C\ar[r,"w"] & \ss^{-r}TA
\end{tikzcd}\end{equation*}
is called a strict exact triangle of weight $r$ if it embeds into a diagram 

\begin{equation*}\begin{tikzcd}
   & &\ss^r C\ar[d,"\phi"]\ar[dr,"\ss^rw"]&  \\
    A\ar[r,"u"] & B\ar[dr,"v"]\ar[r,"v'"] & C'\ar[r,"w'"]\ar[d,"f"] & TA\\& & C\ar[r,"w"] & \ss^{-r}TA
\end{tikzcd}\end{equation*}

where $f$ is an $r$-isomorphism and $f \circ \phi=\eta_r^{\ss^rC}$. This class of triangles allows one to define a \textbf{triangular weight} on the class of exact triangles $\Delta(\cc_\infty)$ in $\cc_\infty$. By this we mean a map 
\begin{equation*}
    w_\infty:\Delta(\cc_\infty) \to \R_{\geq 0}
\end{equation*}
which satisfies the \textbf{weighted octahedral axiom}: Given two exact triangles $\Delta_1$ and $\Delta_2$ one can find two more exact triangles $\Delta_3$ and $\Delta_4$, that fit into an octahedron 

 \begin{equation*}\begin{tikzcd}
   &\Delta_1 & & \Delta_4& &\\
      &  F\ar[r]\ar[d] & 0\ar[r] \ar[d]& TF\ar[r,equals]\ar[d] & TF\ar[d]\\
     \Delta_3&   G\ar[r]\ar[d] & M\ar[r]\ar[d,equals]  & P\ar[r]\ar[d]   & TG \ar[d] \\
    \Delta_2 &   H\ar[r] \ar[d]& M\ar[r]\ar[d]  & N\ar[r]\ar[d]   & TH  \ar[d]\\
     &   TF\ar[r]& 0\ar[r]  & T^2F\ar[r,equal]   & T^2F
    \end{tikzcd}\end{equation*}
    with every square commuting except the bottom right which anti-commutes
    and 
    \begin{equation*}w_\infty(\Delta_3)+w_\infty(\Delta_4)\leq  w_\infty(\Delta_1)+w_\infty(\Delta_2).\end{equation*}
Given $\Delta=   \begin{tikzcd}
        A\ar[r,"f"] & B\ar[r,"g"] & C \ar[r,"h"]& TA
    \end{tikzcd}$ the weight $w_\infty$ is defined as
\begin{equation*}
    w_\infty(\Delta):=\inf\{r: \exists \Delta':=\begin{tikzcd}
        A\ar[r,"f'"] & \ss^{-r_1}B \ar[r,"g'"]& \ss^{-r_2}C\ar[r,"h'"] & \ss^{-r}TA
    \end{tikzcd} \text{ strict exact representing }\Delta\}
\end{equation*}

Where $\Delta'$ represents $\Delta$ if 
\begin{align*}
    f=&[\eta_{-r_1,0}\circ f' ]\\
    g=&[\eta_{-r_2,0} \circ g' \circ \eta_{0,-r_1}]\notag\\
    h=&[\eta_{-r,0}\circ h' \circ \eta_{0,-r_2}]\notag
\end{align*}

There is another triangular weight 
\begin{equation*}
    \bar{w}(\Delta):=\in\ff_{s}\{w_\infty(\ss^{s,0,0,s}\Delta):s\in \R\}
\end{equation*}
with
\begin{equation*}
    \ss^{a,b,c,d}(\begin{tikzcd}A \ar[r]& B \ar[r]& C\ar[r]  & TA\end{tikzcd})=\begin{tikzcd}\ss^aA \ar[r]& \ss^bB\ar[r] & \ss^c C \ar[r]&\ss^dTA\end{tikzcd}
\end{equation*}

and morphisms shifted appropriately by $\eta_{r,s}$'s. Triangular weights allow the construction of (pseudo) metrics called \textbf{fragmentation metrics} on the objects of $\cc$. These metrics are defined as follows: First choose a family $\ff\subset \text{\text{Obj}}(\cc)$ and define

\begin{equation*}
\delta^\ff(A,B)= \inf\{\sum_{i=1}^n \bar{w}(\Delta_i)\}  
\end{equation*}
where $\{\Delta_i\}_{i=1,\hdots,n}$ are exact triangles in $\cc_\infty$ that give an iterated cone decomposition of $A$ from $B$ using $\ff$. That is, a sequence of exact triangles
 \begin{equation*}\begin{tikzcd}[ampersand replacement =\&]
    \Delta_1 : \& X_1\ar[r] \& 0 \ar[r] \& Y_1\ar[r]  \& TX_1\\
    \Delta_2 : \& X_2 \ar[r] \& Y_1 \ar[r] \& Y_2 \ar[r] \& TX_2\\
    \vdots \& \& \vdots \& \vdots\\
    \Delta_n: \& X_n\ar[r]  \& Y_{n-1} \ar[r] \& A \ar[r] \& TX_n                    
\end{tikzcd}\end{equation*} 
where $X_i\in \ff$  $\forall i\neq j$ and $X_j=T^{-1}B$. The symmetrisation of this gives a pseudo metric

\begin{equation*}
    d^\ff(A,B)=\max\{\delta^\ff(A,B),\delta^\ff(B,A)\}
\end{equation*}
Note that this metric can be degenerate, for example if $A$ and $B$ both belong to $\ff$ then $d^\ff(A,B)=0$. It is also often non-finite.

\section{Filtered topological spaces} \label{filttopsec}
We begin by defining what we mean by a filtered topological space, and explore some of the standard notions in topology in the now filtered setting. Note importantly, from the point of view of homotopy theory we choose to work with filtered pointed spaces. Some of the key constructions we explore are notions of filtered smash products, filtered homotopies, and filtered CW complexes. We show that there is an analogue of the CW approximation theorem, this allows us to concentrate on studying filtered CW complexes, and leads to the natural notion of weighted Euler characteristic, which we show is a filtered homotopy invariant. 
\subsection{Definitions}

\begin{defn}\label{filttopdef}
Let $X$ be a pointed topological space, with basepoint $*_X$. We define a \textbf{filtration} $\ff$ on $X$ to be a collection of pointed subspaces of $X$ indexed over the reals,
\begin{equation}
\ff=\{X(r): X(r)\xhookrightarrow{i_r} X\}_{r\in \R} \end{equation}
where $i_r$ is the canonical inclusion map and such that:
\begin{enumerate}
    \item For all $r\in \R_{}$, we have $i_r(*_{X(r)})=*_X$, i.e., they have a common basepoint.
    \item  For all $s<r$ we have $X(s) \subset X(r)$, we denote the inclusion map $i_{s,r}$.
    \item There exists an $r_0\in \R$ such that for all $r< r_0$ we have $X(r)=X_0$, i.e., the filtration is stabilises below to some space $X_0$.

    \item There exists an $r_1\in \R$ such that for all $r>r_1$ we have $X(r)=X$, i.e., the filtration stabilises above to $X$, which we refer to as the \text{total space}.

\end{enumerate}

We will refer to the pair $(X,\ff)$ as a \textbf{filtered topological space}. Note we will often omit the filtration from notation, and simply refer to $X$ as a filtered topological space, the filtration on $X$ will be made clear. One can think of the filtered space as filtered relative to the lower stabilisation $X_0$. It will become clear further in the text why we should view it in this manner. We denote 
\begin{equation}
    \lfloor X \rfloor:=\sup\{r_0: \forall r\leq r_0 \hspace{5pt} X(r)\simeq *\}
\end{equation}
and 
\begin{equation}
    \lceil X \rceil:=\inf\{r_1:\forall r>r_1 \hspace{5pt} X(r)=X\}
\end{equation}
\end{defn}

\begin{eg}
Consider the $n$-sphere $X:=S^n=\{(x_0,\hdots,x_{n}):\sum_{i}x_i^2=1\}$, and functional 
\begin{equation*}h:S^n \to \R\end{equation*}
\begin{equation*}(x_0,\hdots,x_n)\mapsto x_n.\end{equation*}

Up to homeomorphism we have that the level sets of the filtration induced by $h$ are identified with:
\begin{equation*}X(r):=\begin{cases}S^n & r\geq 1\\
D^n & r\in (-1,1)\\
* & r\leq -1\end{cases}\end{equation*}
where $D^n$ is the $n$-disk. Note that up to homotopy we have
\begin{equation*}X(r)\simeq\begin{cases}S^n & r\geq 1\\
* & r<1\end{cases}\end{equation*}.
\end{eg}

\begin{defn}
Given two filtered spaces $(X,\ff_X)$ and $(Y,\ff_Y)$, we define a \textbf{morphism of filtered spaces} $f:(X,\ff_X) \to (Y,\ff_Y)$
to be a collection of continuous maps

\begin{equation}\big\{f(r):X(r) \to Y(r+\lceil f \rceil)\big\}_{r\in \R}\end{equation}

where $\lceil f \rceil\in \R$ is called the \textbf{shift} of $f$. We require that the morphisms $f(r)$ commute with the filtration maps. That is, we require the following to commute for all $s\leq r$:

\begin{equation}\begin{tikzcd}X(r) \arrow[r,"f(r)"] & Y(r+\lceil f \rceil)\\
X(s)\arrow[u,hook,"i_{s,r}^X"]\arrow[r,"f(s)"] & Y(s+\lceil f \rceil) \arrow[u,hook,swap,"i_{s+\lceil f \rceil,r+ \lceil f \rceil}^Y"]\end{tikzcd} \end{equation}

\end{defn}

We can compose morphisms of filtered spaces $f:(X,\ff_X) \to (Y,\ff_Y)$ and $g:(Y,\ff_Y) \to (Z,\ff_Z)$ to obtain a morphsim $g \circ f:(X,\ff_X) \to (Z,\ff_Z)$ with shift
\begin{equation}\lceil g \circ f \rceil = \lceil f \rceil + \lceil g \rceil\end{equation}
This allows us to form a category of filtered spaces which we denote $\ff \text{Top}_*$.
\begin{defn}
We write $\text{Hom}_{\ff \text{Top}_*}((X,\ff_X),(Y,\ff_Y))(a):=\{f:(X,\ff_X) \to (Y,\ff_Y) : \lceil f \rceil = a\}\subset \text{Hom}((X,\ff_X),(Y,\ff_Y))$ for the subset of morphisms with shift term equal to $a$.
\end{defn}
\begin{rmk}
Given $a<b$, we have that $\text{Hom}((X,\ff_X),(Y,\ff_Y))(a)\subset \text{Hom}((X,\ff_X),(Y,\ff_Y))(b)$
\end{rmk}

The category $\ff \text{Top}_*$ comes equipped with a collection of \textbf{shift functors}. Given $a\in \R$ we define $\ss^a:\ff \text{Top}_* \to \ff \text{Top}_*$ by 

\begin{align*}
(\ss^{-a} X)( r)=X(r+a) \\(\ss^{-a}f)(r)=f(r+a)
 \end{align*}

 \begin{rmk}
 Note that $\ss^{a}\circ \ss^{b}=\ss^{a+b}$
and in particular $\ss^{a} \circ \ss^{-a}=\ss^0=1_{\ff \text{Top}_*}$.
 \end{rmk}
\subsection{Products and Wedges }
With the above definition of $\ff \text{Top}_*$ we can define products and coproducts (wedges) of filtered spaces, as well as tensor products (smash products). From now on we will omit the filtration $\ff$ from notation and simply write $X$ to be a filtered space where if there is a specific filtration it will be made clear. We begin with the following:
\begin{defn}
We define the \textbf{naive filtered product} and \textbf{wedge}, $X \times Y$ and $X \vee Y$ by simply taking them levelwise:

\begin{align}(X \times Y)(r):=& X( r) \times Y( r)\\(X \vee Y)(r):= &X( r) \vee Y(r ) \end{align}

 The filtration maps are given by the obvious maps induced from the filtrations on $X$ and $Y$. 
\end{defn}

 There exists a map $\Delta_X:X \to X \times X$ of shift $\lceil \Delta_X\rceil=0$ given by
\begin{equation*}X(r) \xrightarrow{ \Delta(r)}X(r)\times X(r)\end{equation*}
\begin{equation*}x\mapsto (x,x)\end{equation*}
Thus we have $\Delta_X(r)=\Delta_{X(r)}$. Given two morphisms $f,g:X \to Y$ with $\lceil f \rceil = \lceil g \rceil = a$ we can define their product map 

\begin{equation*}(f\times g)(r):(X \times X)( r)=X( r ) \times X( r) \xrightarrow{f( r)\times g( r)}Y( r + a )\times Y( r+a) = (Y\times Y)(r+a)\end{equation*}

The shift remains the same, $\lceil f\times g \rceil = a$. Though this seems to behave in the manner one would expect, we now define a better behaved product of filtered spaces. It is as follows:
 \begin{defn}
     Given $X$ and $Y$ in $\ff \text{Top}_*$, we define their \textbf{filtered product} $X \bar{\times} Y$ by:
     \begin{equation}
         (X \bar{\times} Y)(r):= \bigcup_{s+t=r}X(s) \times Y(t)
     \end{equation}
     where the union is taken inside of the product of the total spaces $X \times Y$. And the filtration maps are the obvious inclusions. 
     \end{defn}

Notice that there is an `eternal' (we will define this shortly) copy of $X \vee Y$ contained in $X \bar{\times}Y$. One can think of this filtered space as $X \times Y$ filtered relative $X\vee Y$. Indeed, if one thinks to define a similar wedge sum via unions, $X \bar{\vee} Y$, we simply obtain: 

\begin{equation*}
    (X \bar{\vee} Y) (r)= X \vee Y
\end{equation*}
 the wedge sum of the total spaces. 

Given $f: X \to Y $ and $f': X' \to Y'$ we define $f \bar{\times }f': X \bar{\times} X' \to Y \bar{\times} Y'$ as follows. $(f \bar{\times}f')(r)$ is given by first considering $f$ and $f'$ as morphisms of the total space, i.e., we consider $f(\lceil X \rceil)$ and $f'(\lceil X'\rceil)$, then we restrict this map to the subset of $X \times Y$ (the product of the total spaces) to $(X \bar{\times }X')(r)$

\begin{equation}
(f \bar{\times}f')(r):= f(\lceil X\rceil)\times f'(\lceil X' \rceil))|_{(X\bar{\times}X')(r)} \end{equation}

\begin{rmk}
    There does not exist a natural diagonal map $X \to X \bar{\times}X$. Any such map would need to be such that $\bar{\Delta}:X(r) \to \bigcup_{s+t=r+k}X(s) \times X(t) $ where $k=\lceil \bar{\Delta}\rceil$ is some shift. Take for example the unit interval $I=[0,1]$ filtered with $I(r)=[0,r]$ with $I(r)=*$ for $r\leq 0$ and $I(r)=[0,1]$ for $r\geq 1$. Then $r\in I(r)$ but $(r,r)\notin (I \bar{\times}I)(r)$ thus we must shift by some $k$ to $(I \bar{\times}I)(r+k)$. But this $k$ will depend on $r$.
\end{rmk}

\begin{figure}[H]
    \centering
    \begin{subfigure}[b]{0.49\linewidth}
         \includegraphics[width=\textwidth,trim={0cm 0cm 0cm 0cm},clip]{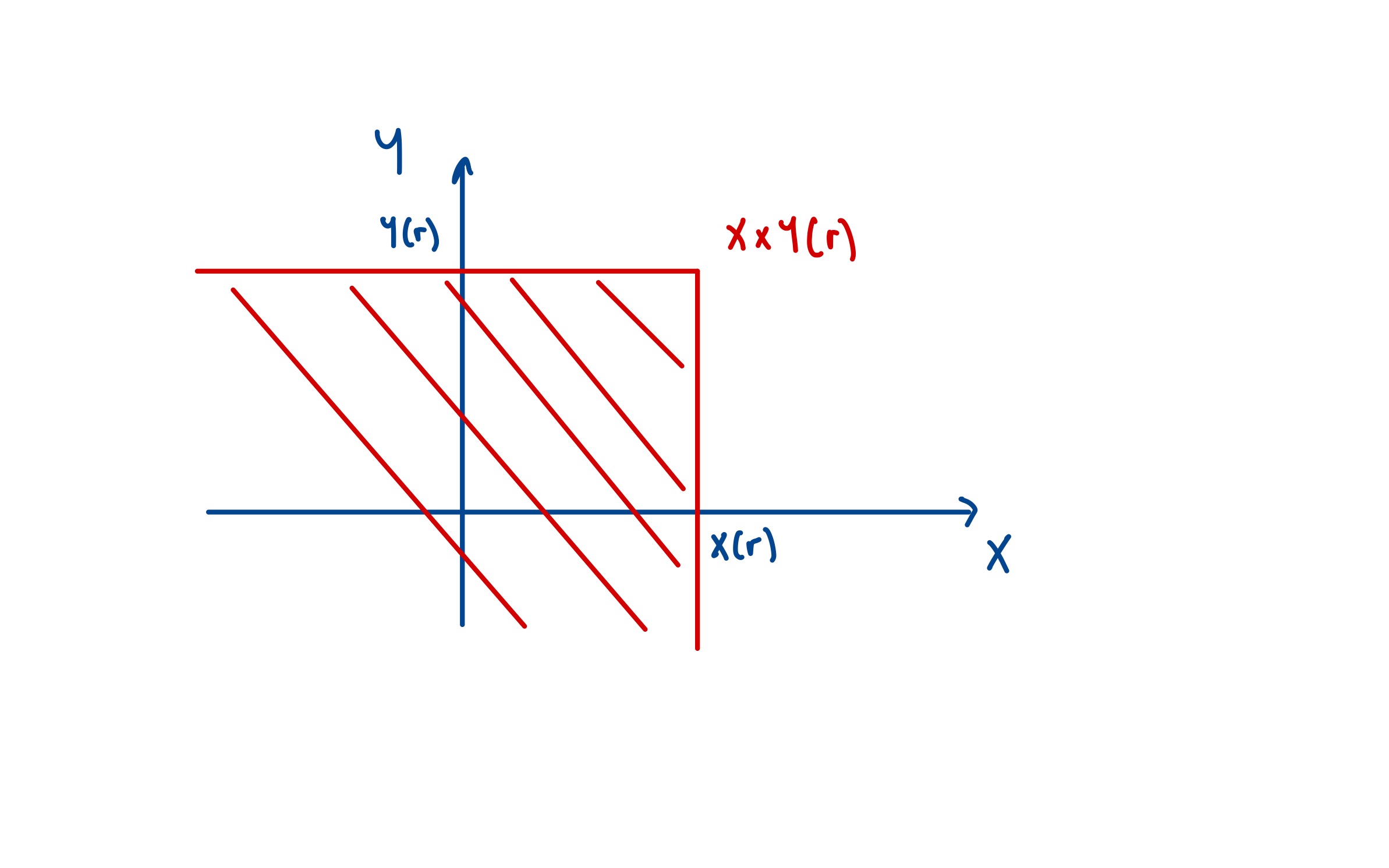}
    \end{subfigure}
    \begin{subfigure}[b]{0.49\linewidth}
         \includegraphics[width=\textwidth,trim={0cm 0cm 0cm 0cm},clip]{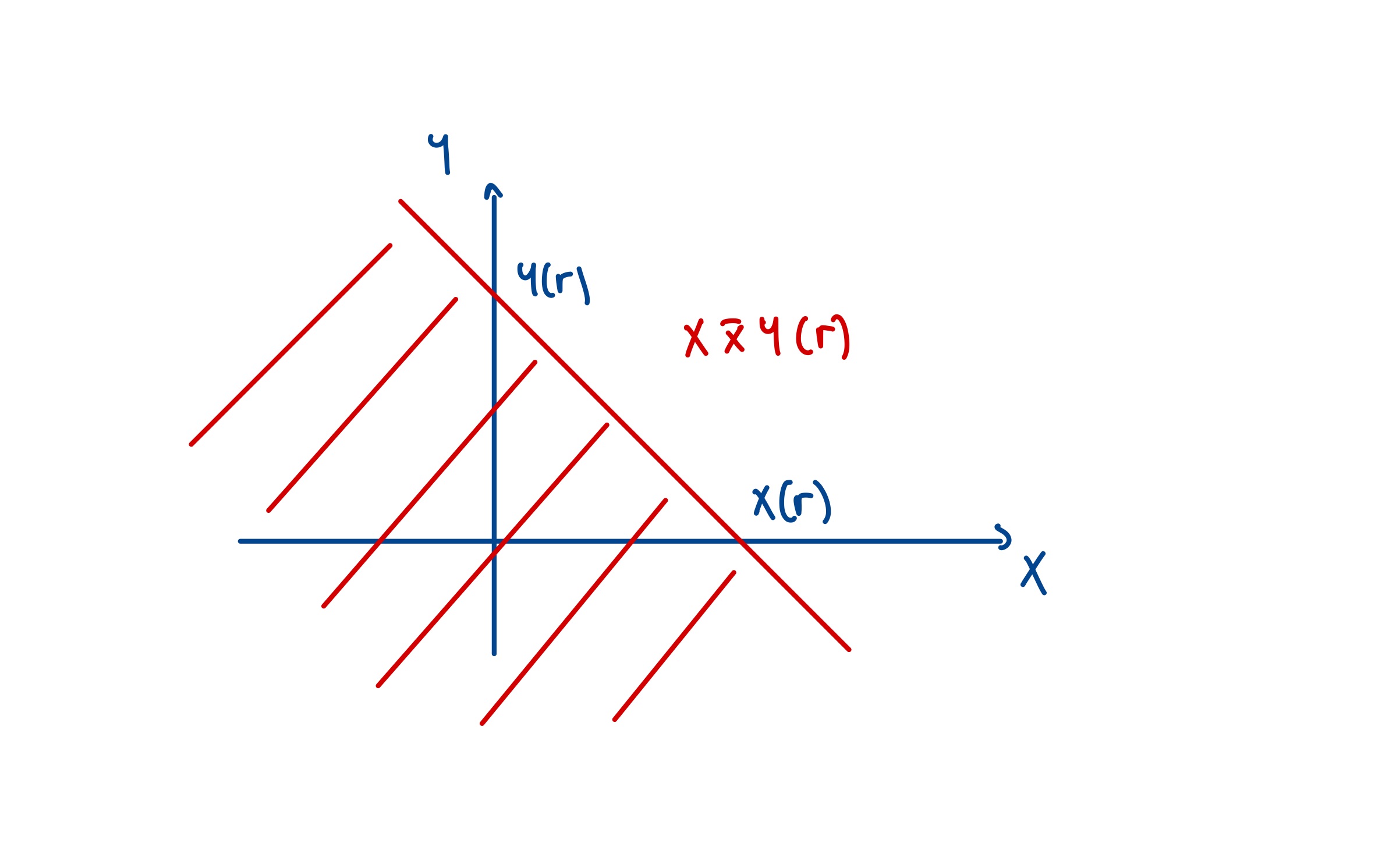}
    \end{subfigure}
   \caption{Diagrams representing the naive product $X \times Y$ (left) and the product $X\bar{\times}Y$ (right).}
    \label{products}
\end{figure}

\begin{defn}
    The \textbf{naive filtered smash product} is defined by 
\begin{equation}(X \Smash Y)(r):= X( r)\Smash Y(r) \end{equation}
On morphisms, it is given simply by the levelwise smash product for morphisms.
\end{defn}

Note that this is simply the quotient of the naive product by the naive wedge;  $(X \Smash Y)(r)=(X\times Y)(r)/(X \vee Y)(r) $. 
\begin{defn}
    We define the \textbf{filtered smash product} $\bar{\Smash}$ by:

    \begin{equation}
        (X \bar{\Smash}Y)(r):= (X \bar{\times} Y)(r)/X \vee Y= \bigcup_{s+t=r} X(s) \Smash Y(t).
    \end{equation}
\end{defn}
On morphisms, we take the product of morphisms $f\bar{\times}f'$, then look to the induced map on the quotient space.
\begin{prop}
    The filtered smash product commutes with shifting:
   \begin{equation}
       \ss^a(X \bar{\Smash}Y) = \ss^aX \bar{\Smash}Y = X \bar{\Smash} \ss^a Y.
   \end{equation} 
   \begin{proof}
       This is a simple calculation:
       \begin{align*}
           (X\bar{\Smash} \ss^aY)(r)=& \bigcup_{s+t=r}X(s) \Smash \ss^aY(t)\\
           =& \bigcup_{s+t=r}X(s) \Smash Y(t-a)\\
           =& \bigcup_{s} X(s) \Smash Y((r-a)-s)\\
           =& \bigcup_{s+t=r-a}X(s) \Smash Y(t)\\
           =& (X\bar{\Smash}Y)(r-a)\\
           =& \ss^{a}(X\bar{\Smash}Y)(r)
       \end{align*}
       And similarly it is easy to verify $\ss^a (X \bar{\Smash}Y)= \ss^a X \bar{\Smash}Y$.
   \end{proof}
\end{prop}

The properties of these smash products will soon be explored more, and will be key to proving Theorem \ref{Mainiso}.

\subsection{Homotopy}

Given a filtered space $X$ we define its \textbf{filtered cylinder} to be the filtered space with
\begin{equation}\text{Cyl}(X)(r):= \text{Cyl}\big(X(r)\big)\end{equation}
Note this can be viewed as $X \times I$ where $I$ has the \textbf{trivial filtration}:
\begin{equation*}
    I(r):=
        [0,1]\hspace{8pt} \forall r\in \R.
\end{equation*}

We similarly define the \textbf{reduced filtered cylinder} by
\begin{equation}\text{Cyl}_*(X)(r):= \text{Cyl}_*\big(X(r)\big).\end{equation}

This can be viewed as $X \Smash I_+ $ with $I_+$ having the trivial filtration (where the subscript denotes adjoining of a disjoint basepoint).

\begin{defn}
Let $f,g:X \to Y$ be morphisms with shift $\lceil f \rceil = \lceil g \rceil = a$, we define a \textbf{ homotopy} from $f$ to $g$, to be a map 
\begin{equation}h:\text{Cyl}_*(X) \to Y\end{equation}

such that 
\begin{equation}\begin{tikzcd} X \arrow[r,"i_0"]\arrow[dr,swap,"f"]& \text{Cyl}_*(X) \arrow[d,"h"]& X\arrow[l,swap,"i_1"]\arrow[dl,"g"]\\
& Y &\end{tikzcd}\end{equation}
commutes, where $i_0,i_1$ are the obvious inclusion maps of shift zero. In particular $\lceil h \rceil = a$.

\end{defn}

\begin{rmk}
This defines an equivalence relation, $\sim_a$, on $\text{Hom}(X,Y)(a)$, we denote 
\begin{equation}[X,Y](a):=\text{Hom}(X,Y)(a)/\sim_a\end{equation}
\end{rmk}

\begin{defn}
    We call two filtered spaces $X$ and $X'$ \textbf{$r$-filtered homotopy equivalent}, if there exist maps $f:X \to X'$ and $g:X' \to X$ with $\lceil f\rceil=r=-\lceil g \rceil$ such that $f\circ g \sim 1_{X'}$ and $g \circ f \sim 1_X$. Note that we will sometimes refer to a $0$-homotopy equivalence simply as a homotopy equivalence or filtered homotopy equivalence.
\end{defn}

Consider the $k$-sphere with filtration concentrated at $0\in \R$ so that 
\begin{equation}
    S^k(r):=\begin{cases}
        S^k & r\geq 0\\
        * & r<0
    \end{cases}
\end{equation}
We will call this the \textbf{zero filtration on $S^k$}. We define the \textbf{$k$-th filtered homotopy group} of a filtered space $X$ to be 
\begin{equation}\pi_k^r(X):=[S^k,X](r)\end{equation}

\begin{rmk}
By the definition of filtration $(S^k)( r)$, it can easily be verified that 
\begin{equation*}\pi_k^r(X)=[S^k,X](r)=\pi_k(X(r)).\end{equation*}
Indeed, $[S^k,X](r)$ are homotopy classes of maps from $S^k$ to $X$ with shift $r$. Such maps are given by families

\begin{equation*}
    \begin{tikzcd}
        S^k(b)\ar[r,"f(b)"] & X(b+r)\\
        S^k(a)\ar[u,"i_{a,b}^{S^k}"] \ar[r,"f(a)"]& X(a+r)\ar[u,swap,"i_{a+r,b+r}^X"]
    \end{tikzcd}
\end{equation*}
For all $0\leq a\leq b$ we have $i_{a,b}^{S^k}=1_{S^k}$, hence such a map $f$ is given by
\begin{equation*}
    \begin{tikzcd}
        & X(b+r)\\
        S^k\ar[r,swap,"f(a)"]\ar[ur,"f(b)"]& X(a+r)\ar[u,swap,"i_{a+r,b+r}^X"]
    \end{tikzcd}
\end{equation*}
Since $i_{a,b}^X$ are inclusions we have $\text{Im}(f(a))=\text{Im}(f(b))$ inside of the total space $X$. A homotopy (in the filtered setting) from such an $f$ to some $f'$ (with $f'$ also of shift $r$) is a map $h:\text{Cyl}_*(S^k) \to X $, which explicitly is a family of maps

\begin{equation*}
  \begin{tikzcd}  \text{Cyl}_*(S^k)(b)\ar[r,"h(b)"] & X(b+r)\\
    \text{Cyl}_*(S^k)(a)\ar[u,"i_{a,b}^{\text{Cyl}_*(S^k)}"]\ar[r,"h(a)"] & X(a+r)\ar[u,"i_{a+r,b+r}^X"]
    \end{tikzcd}
\end{equation*}
Again, when $0\leq a\leq b$ we realise that $i_{a,b}^{\text{Cyl}_*(S^k)}=1_{\text{Cyl}_*(S^k)}$. Which implies the following
\begin{equation*}
  \begin{tikzcd}  
  & X(b+r)\\
    \text{Cyl}_*(S^k)\ar[ur,"h(b)"]\ar[r,"h(a)"] & X(a+r)\ar[u,swap,"i_{a+r,b+r}^X"]
    \end{tikzcd}
\end{equation*}
and hence $\text{Im}(h(a))=\text{Im}(h(b))$ inside of the total space $X$. Putting this all together, we see that two (filtered) morphisms $f,f':S^k \to X$ are homotopic iff there is a homotopy between the limiting maps $\lim_a f(a): S^k \to X(\infty)=X$ and $\lim_a f'(a):S^k \to X(\infty)=X$ given by $\lim_a h(a)$.
\end{rmk}

We will want to work with filtrations that are suitably `nice', we therefore restrict our attention to pairs $(X,\ff_X)$ where the filtration $\ff_X$ has the additional property that:
\begin{itemize} \label{assump}
    \item For all $k\in \N$, there exist countably many $r\in \R$ such that for any $\epsilon>0$ we have  $\pi_k^r(X) \neq \pi_k^{r-\epsilon}(X)$.
\end{itemize}

This property tells us that the homotopy type of the filtered space can only change at countably many points. We refer to these points as \textbf{spectral points} of the filtration and denote the set of spectral points of a filtered space by $\text{Spec}(X)$. From now on we will assume our filtered spaces to have this additional property and refer to $\ff \text{Top}_*$ as the category consisting of objects with this additional assumption.

\begin{defn}
We denote the \textbf{suspension} of a filtered space $X$ to be $\S X$, defined by
\begin{equation}\S X(r):= \S \big(X(r)\big)\end{equation}
This can be seen to be given by $\S X= S^1 \Smash X$ with $S^1$ having trivial filtration.
\end{defn}
Note that suspension of the trivially filtered $k$-sphere $S^k$, gives the trivially filtered $(k+1)$-sphere, $S^{k+1}$. Similarly, suspension of a $k$-sphere with zero filtration will give the $(k+1)$-sphere with zero filtration (after identifying $\S * \simeq *$). 
\begin{rmk}
    Given a general filtered space $X$, we may define a corresponding \textbf{cut off} filtered space  $\text{Cut}_a(X)$, given by
    \begin{equation}
     \text{Cut}_a(X)(r):=   \begin{cases}
            X(r) & r\geq a\\
            * & r<a
        \end{cases}
    \end{equation}
    The zero filtration on any given space can thus be thought of as the cut off at $a=0$ of the trivial filtration. 
\end{rmk}

\begin{prop}
    Denote by $S^1_0$, the zero filtered $1$-sphere, and by $S^1_\infty$ the trivially filtered sphere, let $X$ be any filtered space, then 
    \begin{equation}
        S_\infty^1 \Smash X \simeq S^1_0 \bar{\Smash} X. 
    \end{equation}
are zero homotopy equivalent.
\begin{proof}
    Notice that we have the following
    \begin{align*}
     (S^1_0 \bar{\times} X)(r) =& \bigcup_{t}S^1(t)\times X(r-t)\\
     =&
         S^1 \times X(r) \cup * \times X
    \end{align*}
    Indeed, when $t<0$, $S^1_0(t)=*$. Thus 
    \begin{equation*}
        (S^1_0\bar{\Smash}X)(r)=S^1 \times X(r) \cup * \times X/(S^1 \vee X)
    \end{equation*}
    But this is clearly equivalent to $S^1 \times X(r)/S^1 \vee X(r)= S^1\Smash X(r)$.
\end{proof}
\end{prop}
Thus, we can think of suspension in terms of either $\Smash$ or $\bar{\Smash}$. A similar statement holds for higher spheres.

\begin{figure}[H]
    \centering
    \includegraphics[width=0.5\linewidth]{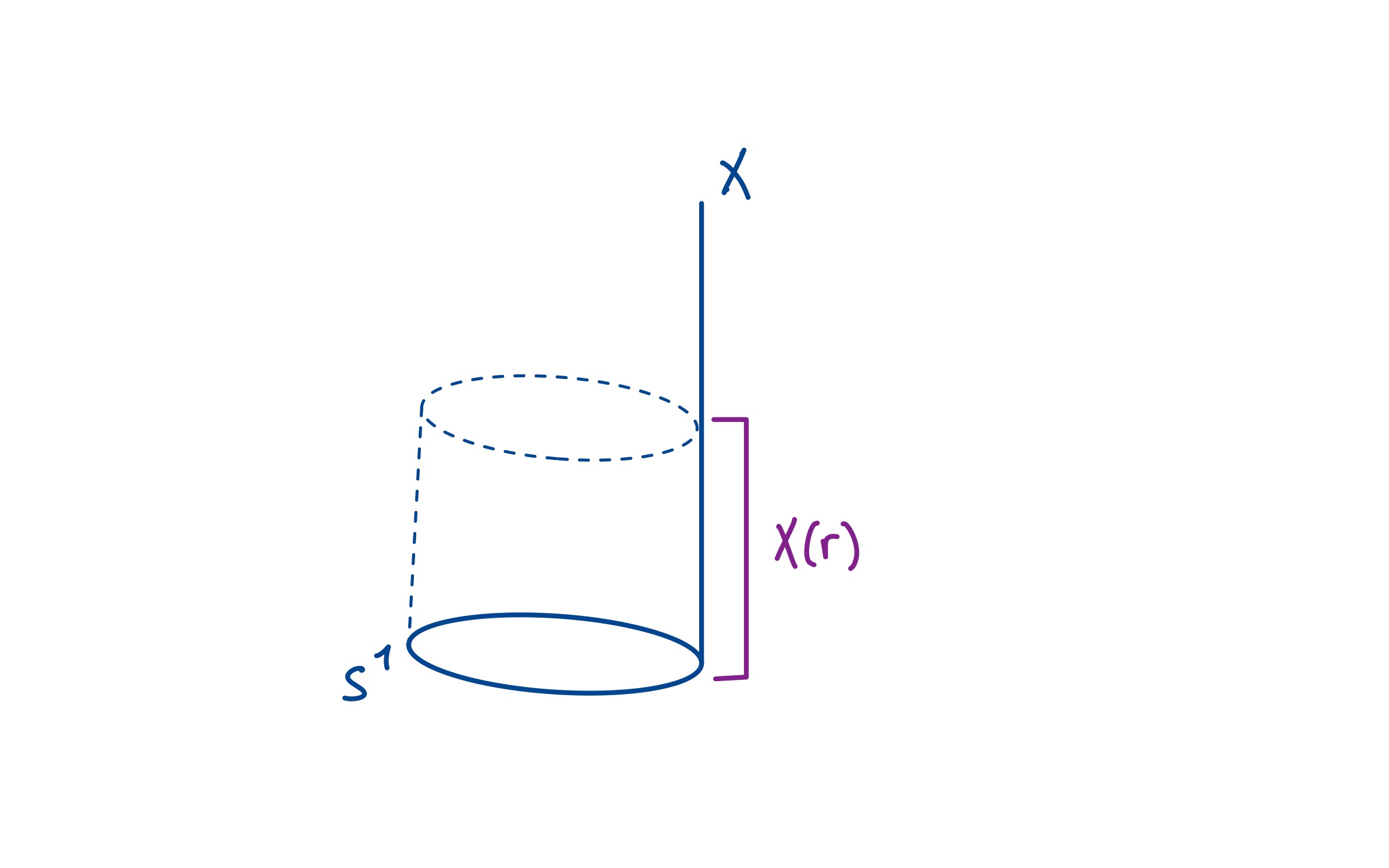}
    \caption{A depiction of $(S^1_0 \bar{\times} X)(r) $ showing the eternal subcomplex $S^1 \vee X$.}
    \label{fig:enter-label}
\end{figure}

\subsection{Filtered CW-complexes and weighted Euler characteristic}
In order to construct homotopical notions in the filtered setting we need to understand what a filtered CW complex should be and what filtered weak homotopy equivalences are. 

\begin{defn}
A \textbf{filtered CW complex} is a filtered space $X$ such that the levelsets are all sub-CW-complexes and the inclusion maps are cellular inclusions. A morphism of filtered CW complex is a morphism of filtered topological spaces such that each level is cellular.
\end{defn}

\begin{rmk}
    We refer to the cells in the subcomplex $X_0$ of some filtered CW complex as \textbf{eternal cells}, and the subcomplex itself as the \textbf{eternal subcomplex}. This is because the cells are never created, they exist for all $X(r)$.  One can think of a filtered CW complex $X$ to be filtered relative to the eternal subcomplex. 
\end{rmk}

\begin{lemma}[Filtered CW approximation]\label{CWspacefilt}
Given any filtered space $X\in \ff \text{Top}_*$, there exists a filtered CW space $\bar{X}$ and a map $f: \bar{X} \to X$ of shift zero, such that $f(r): \bar{X}( r) \to X(r)$ is a weak equivalence for all $r$, and moreover the following homotopy commutes for all $r \leq s$
\begin{equation}
    \begin{tikzcd}
        \bar{X}(s)\ar[r,"f(s)"] & X(s)\\
        \bar{X}(r)\ar[r,"f(r)"]\ar[u,"i_{r,s}^{\bar{X}}"] & X(r).\ar[u,"i^X_{r,s}"]
    \end{tikzcd}
\end{equation}

\end{lemma}

The proof of this lemma relies on the following known result (see \cite{Ma1} Chapter 10, section 6):

\begin{lemma}
    Given a pair $(X,A)$ and a CW approximation of $A$, $\bar{A}\xrightarrow{\alpha_A} A$, then there exists a CW approximation of $X$, $\bar{X}\xrightarrow{\alpha_X} X$ such that $\bar{A}$ is a subcomplex of $\bar{X}$ and $\alpha_X$ restricts to $\alpha_A$ on $\bar{A}$.
\end{lemma}

\begin{proof}[Proof of Lemma \ref{CWspace}]
We will build such an $\bar{X}$, as follows: Firstly, we may assume by the assumption \ref{assump} that $\text{Spec}(\ff_X)=\{r_i:i\in \N\}$. Set $r_0:=\lfloor X \rfloor >-\infty$ such that $X(r)=X_0$ for all $r<r_0
$. Chose some $\bar{X}(r):=\overline{X_0}$ for all $r<r_0$. We now apply the above lemma to the pair $(X(r_0),X(r))$ for all $r<r_0$, to obtain a homotopy commutative diagram
\begin{equation*}
    \begin{tikzcd}
        \bar{X}(r_0)\ar[r,"\alpha(r_0)"] & X(r_0)\\
       \bar{X}(r)=\overline{X_0}\ar[r,"\alpha(r)"]\ar[u,"i_{r,r_0}^{\bar{X}}"] & X(r)\ar[u,swap,"i^X_{r,r_0}"]
    \end{tikzcd}
\end{equation*}

We then set the approximations $(\bar{X}(r)\xrightarrow{\alpha(r)} X(r)):=(\bar{X}(r_0) \xrightarrow{\alpha(r_0)} X(r_0))$ for all $r\in [r_0,r_1)$, and take $i_{r,s}^{\bar{X}}$ to be the identity for all $r,s\in [r_0,r_1)$. Next we apply the above lemma iteratively over each $r_i\in \text{Spec}(\ff_x)$, to obtain diagrams

\begin{equation*}
    \begin{tikzcd}
       \bar{X}(r_{i+1}) \ar[r,"\alpha(r_{i+1})"]&X(r_{i+1}) \\
        \bar{X}(r)\ar[r,"\alpha(r_{i-1})"]\ar[u,"i_{r,r_{i+1}}^{\bar{X}}"]& X(r)\ar[u,swap,"i_{r,r_{i+1}}^X"] 
    \end{tikzcd}
\end{equation*}

with $r\in [r_i,r_{i+1})$ and approximations $(\bar{X}(r)\xrightarrow{\alpha(r)} X(r))=(\bar{X}(r_{i})\xrightarrow{\alpha(r_i)} X(r_i))$. Setting $i_{r,s}^{\bar{X}}$ to be the identity for all $r<s\in [r_{i+1},r_{i+2})$. Finally assume $r<s \in \R$ with $r\in [r_i,r_{i+1})$ and $s\in [r_j,r_{j+1})$, we set 
\begin{equation*}
    i_{r,s}^{\bar{X}}:=i_{r_{j-1},r_{j}}^{\bar{X}}\circ \hdots \circ i_{r_{i+1},r_{i+2}}^{\bar{X}}\circ i_{r,r_{i+1}}^{\bar{X}}
\end{equation*}
Thus, $(\bar{X},\bar{\ff})$ is the desired filtered CW complex.
\end{proof}

\begin{rmk}
The filtered CW complex formed is not unique, but is unique up to filtered homotopy equivalence.
\end{rmk}
\begin{defn}
To every cell of the filtered CW complex $a\subset X$ we can associate a \textbf{cell size} $w(a)\in \text{Spec}(\ff) \cup \{-\infty\}$ given 
    \begin{equation}
        w(a)=\inf\{r_i: a \subset X(r_i)\}
    \end{equation}
    where $w(a)=-\infty$ if $a$ is an eternal cell.
\end{defn}

\begin{rmk}
    Given two filtered CW complexes $X$ and $Y$ the product $X\bar{\times}Y$ is a filtered CW complex with one $k$-cell $c=a\times b$ for every $l$-cell, $a$, and $m$-cell, $b$, in $X$ and $Y$ respectively, such that $l+m=k$. Furthermore, the induced filtration is such that $X \bar{\times} Y(r)$ consists of all cells $a \times b$ such that 
    \begin{equation}
        w(a\times b)= w(a) + w(b) \leq r.
    \end{equation}
\end{rmk}

\begin{prop}
    The smash product $\bar{\Smash}$ on filtered CW-complexes is distributive over the wedge:
    \begin{equation}
        A \bar{\Smash}(X \vee Y) = (A \bar{\Smash} X )\vee (A \bar{\Smash} Y)
    \end{equation}
    Indeed on the total space we know the usual smash product of spaces is distributive over usual wedge product. Thus, we just need to check that the filtrations are equivalent, but this is clear. A cell $c$ in $A \bar{\Smash}(X \vee Y)$ is either a cell of the form $a \Smash x$ or $a \Smash y$ for some cells $a \subset A$, $x\subset X$ and $y\subset Y$. We may assume it to be of the form $a \Smash x$. The weight of this cell is thus given by $w(a \Smash x)= w(a \times x)=w(a)+w(b)$. Note that if $a$ or $x$ is eternal, then we treat $-\infty +k=-\infty$, i.e., $a \times x$ is eternal. The filtration on $(A \bar{\Smash}X) \vee (A \bar{\Smash}Y)$ is such that a cell belonging to the component of $A \bar{\Smash}X$ will have weight $w(a \bar{\Smash}x)=w(a \times x)=w(a)+w(x)$. Thus the filtrations on the two (total) CW complexes agree.
\end{prop}

\begin{rmk}\label{sizeofsmashmorphism}
    If $f: X \to Y$ and $f':X' \to Y'$ are morphisms of filtered CW complexes, then 
    \begin{align*}
        w\big( (f\bar{\Smash} f')(a \bar{\Smash}b)\big) = w(a)+\lceil f\rceil + w(b) + \lceil f' \rceil= w(a\times b) + \lceil f \rceil + \lceil f' \rceil 
    \end{align*}
    Thus $\lceil f \bar{\Smash}f'\rceil = \lceil f \rceil + \lceil f' \rceil$.
\end{rmk}

\begin{defn}
   Given a filtered CW complex $\bar{X}$, its \textbf{size polynomial} is given by
    \begin{equation}
       \lambda( \bar{X})_{\text{CW}}:=\sum_{a\in \text{Cells}_*(X)} t^{w(a)} 
    \end{equation}
 Where $\text{Cells}_*(X)$ denotes the cells of $X$ without including the basepoint $*$. 
Furthermore, we consider the derivative 

\begin{equation}
  \Lambda_{\text{CW}}(\bar{X}):=  \frac{d}{dt}\big(\lambda_{\text{CW}}(\bar{X})\big)=\sum_{a\in \text{Cells}_*(\bar{X})}w(a) \cdot t^{w(a)-1}
\end{equation}

\end{defn}

\begin{rmk}
    If $a\in \text{Cells}_*(X)$ is such that $w(a)=-\infty$, i.e., $a\in X(r)$ for all $r\in \R$, then we set $t^{-\infty}=0$, i.e., do not include this cell in the count. 
\end{rmk}

Evaluation at $t=1$ of the size polynomial gives the number of cells in $\bar{X}$ (excluding the basepoint), which we denote by 
 \begin{equation}
|\bar{X}|_0=\text{ev}_{t=1}\big(\lambda_{\text{CW}}(X)\big)
 \end{equation}
 and evaluation at $t=1$ of its derivative gives the weighted count of cells denoted by:

\begin{equation}
|\bar{X}|=\text{ev}_{t=1}\big(\Lambda_{\text{CW}}(X)\big).
\end{equation}

For a general filtered space $X$ we define
    \begin{align}
        |X|_0&:=\inf\{|\bar{X}|_0:\bar{X} \text{ is a filtered CW approximation of } X\}\\
        |X|&:=\inf\{|\bar{X}|:\bar{X} \text{ is a filtered CW approximation of } X\}.
    \end{align}

\begin{eg}
    Take $S^k$ to be the filtered $k$-sphere with filtration given by
    \begin{equation*}
        S^k(r)=\begin{cases}
            * & r<t\\
            S^k & r\geq t
        \end{cases}
    \end{equation*}
    then $|S^k|_0=1$ and  $|S^k|=t$ as we can view $S^k$ as the CW complex with a zero-cell (the basepoint) and one $k$-cell which attaches at level $t$. 
\end{eg}

\begin{eg}\label{torus example}
    Consider a torus $T$ with height function such that we have a filtered spaces with 
    \begin{equation*}
        T(r)\simeq\begin{cases}
            * & r<a\\
            S^1 & a\leq r <b\\
            S^1 \vee S^1 & b\leq r < c\\
            T & r\geq c
        \end{cases}
    \end{equation*}
    This CW complex for a torus has one $0$-cell (the basepoint), two $1$-cells and one $2$-cell. We can attach each of these cells at appropriate times to find $|T|_0=3$ and  $|T|=a +b +c$, since we can take one $1$-cell with $w=a$, one $1$-cell with $w=b$ and a two cell with $w=c$. 

  \begin{figure}[H]
    \centering
    \begin{subfigure}[b]{0.49\linewidth}
         \includegraphics[width=\textwidth,trim={0cm 0cm 0cm 0cm},clip]{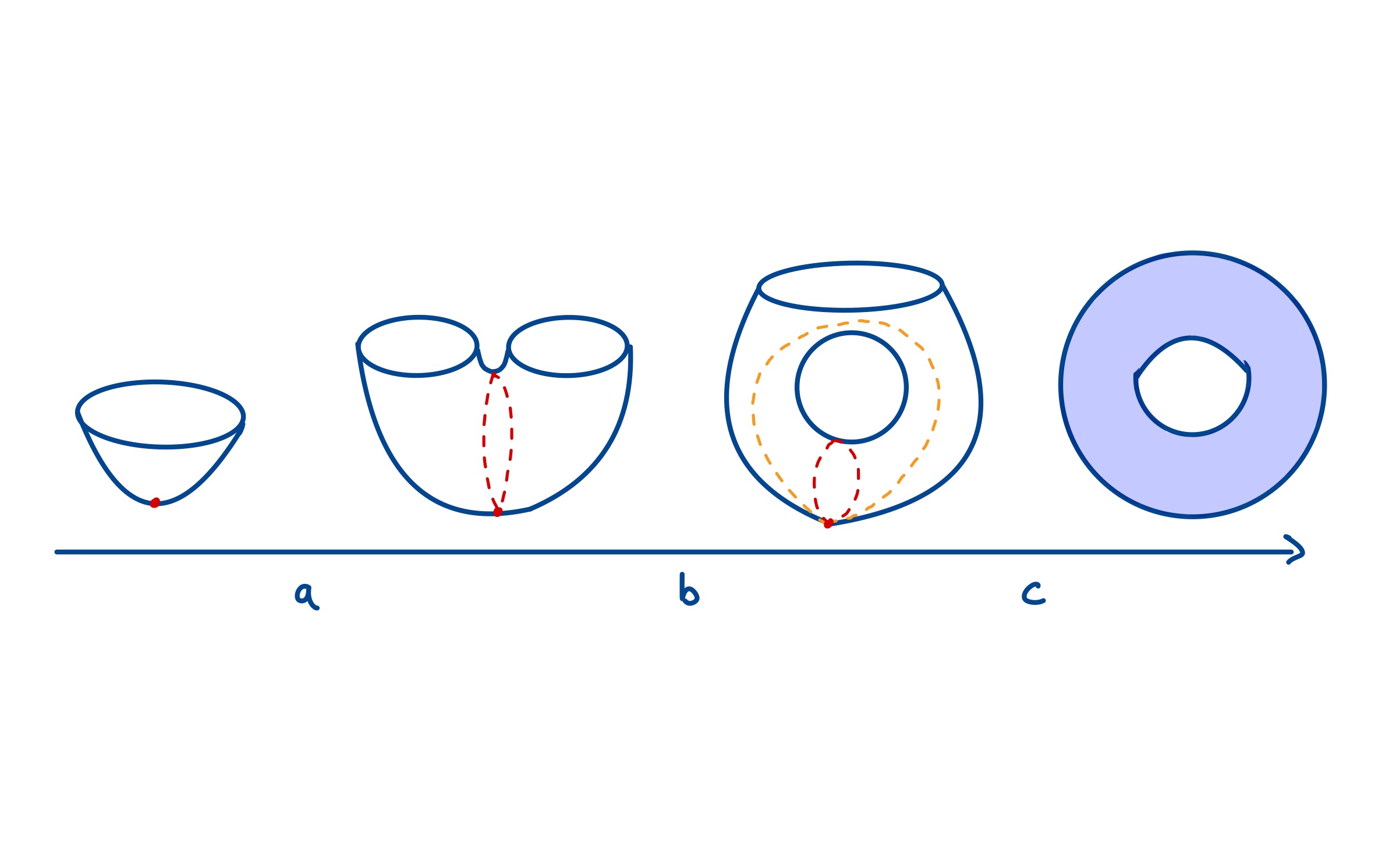}
    \end{subfigure}
    \begin{subfigure}[b]{0.49\linewidth}
         \includegraphics[width=\textwidth,trim={5cm 5cm 5cm 5cm},clip]{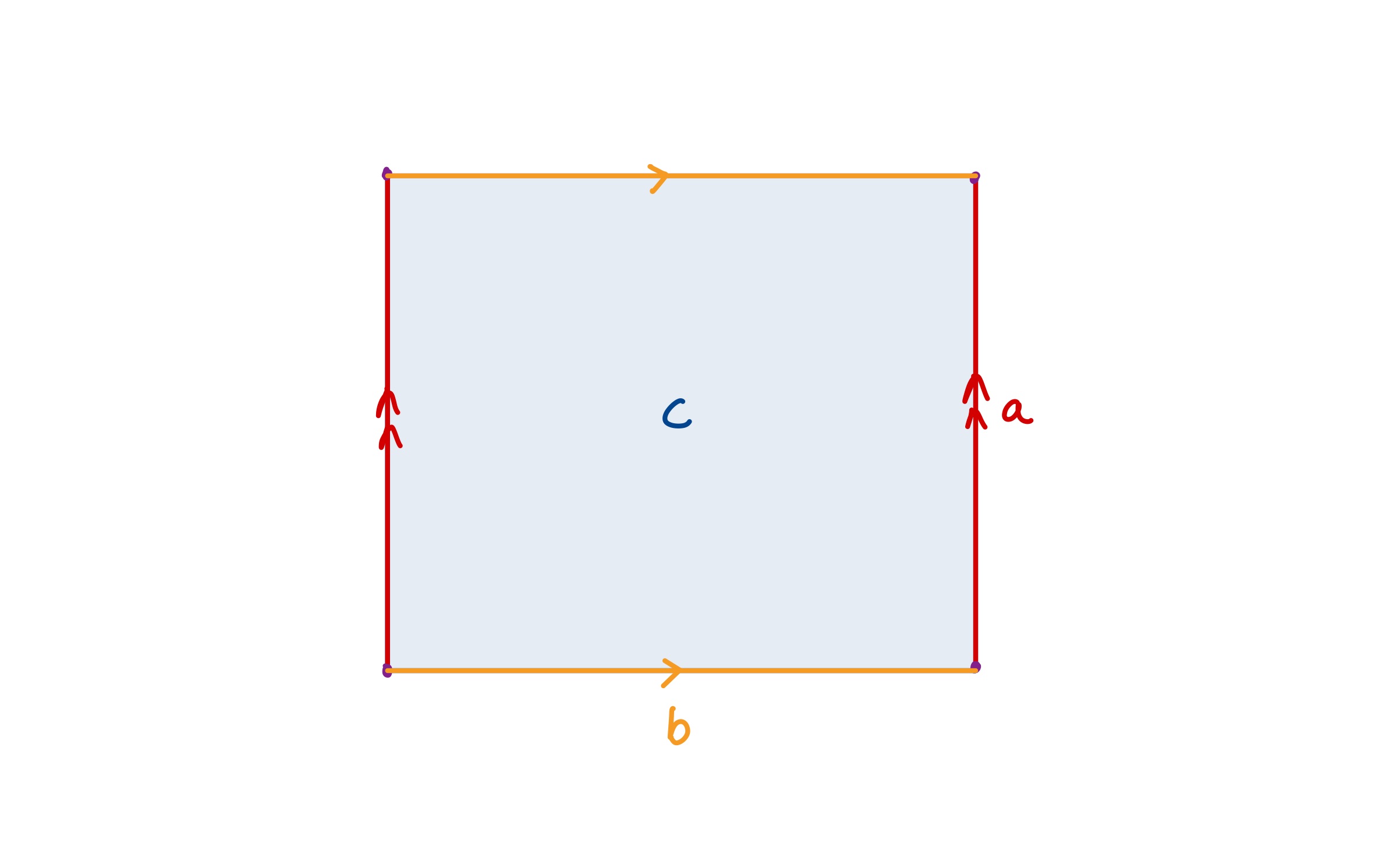}
    \end{subfigure}
   \caption{A generic picture of $T(r)$ in blue for various values of $r$ with a deformation retract to a corresponding CW complex coloured in red/orange (left). Along with a depiction of a filtered CW complex with cell weights labeled which is filtered weak equivalent (right).  }
    \label{torus}
\end{figure}

\end{eg}

The size polynomial and, in turn, its derivative depend on the `filtered cell decomposition'. We will see, however, that there is a weighted version of the Euler characteristic that is a filtered homotopy invariant. Note that a similar object has been studied previously by various authors (for example see \cite{DG} and \cite{HL}).

\begin{defn}
    Given a filtered CW complex, we define its \textbf{Euler polynomial} as 

    \begin{equation}
        \hat{\chi}_{\text{CW}}(X)= \sum_{a\in \text{Cells}_*(X)} (-1)^{|a|}\cdot t^{w(a)}
    \end{equation}
    where $t$ is some formal variable and $|c|$ denotes the dimension of the cell $c$. 
\end{defn}

\begin{rmk}
    Again, we treat cells $a\in \text{Cells}_*(X)$ with $w(a)=-\infty$ to contribute $t^{-\infty}=0$ to this count. We can think of this as a Euler characteristic relative to the `eternal' cells of the filtered complex. 
\end{rmk}
\begin{rmk}
    Evaluation of $\hat{\chi}_{\text{CW}}(X)$ at $t=1$ is precisely the Euler characteristic of (the total space) $X$, minus $1$, as we do not include the basepoint:
    \begin{equation}\text{ev}_{t=1}\big( \hat{\chi}_{\text{CW}}(X)\big)=\chi(X)-1
    \end{equation}
    Note that this is often referred to as the reduced Euler characteristic. Moreover, if we set \begin{equation}
 \hat{\chi}_{\text{CW}}^{\leq r}(X):=\sum_{a\in \text{Cells}_*(X)\text{ with } w(a) \leq r}(-1)^{|a|}\cdot t^{w(a)}
    \end{equation}
    then $\hat{\chi}_{\text{CW}}^{\leq r}(X)$ evaluated at $t=1$ recovers $\chi(X(r))-1$.
\end{rmk}
\begin{prop}\label{propeul}
    If $X\simeq X'$ are two filtered CW complexes, then $\chi^{\leq r}_{\text{CW}}(X)=\chi^{\leq r}_{\text{CW}}(X')$.
    \begin{proof}
        By assumption, the homotopy type of filtration levels can only change at finitely many points, and if $X\simeq X'$, then $\text{Spec}(X)=\text{Spec(X')}$. Furthermore, for each $s\in \R$ we have $X(r)\simeq X'(r)$, hence 
        \begin{equation*}
   \text{ev}_{t=1}\big(\hat{\chi}^{\leq r}_{\text{CW}}(X)\big)=\chi(X(r))-1=\chi(X'(r))-1=\big(  \text{ev}_{t=1}\hat{\chi}^{\leq r}_{\text{CW}}(X)\big).
        \end{equation*}
        Set $D^{\leq r}(X,X'):=\hat{\chi}^{\leq r}_{\text{CW}}(X)-\hat{\chi}^{\leq r}_{\text{CW}}(X')$, then $\text{ev}_{t=1}D^{\leq r}(X,X')=0$. If we assume that $D^{\leq r}(X,X')\neq 0$, then there must exist some $\pm c \cdot  t^s$ term (for $s\leq r$). But this would imply that $\text{ev}_{t=1}\hat{\chi}^{\leq s}_{\text{CW}}(X) \neq \text{ev}_{t=1}\hat{\chi}^{\leq s}_{\text{CW}}(X')$, i.e., $\chi(X(s)) \neq \chi(X'(s))$ and thus $X(s)$ is not homotopy equivalent to $ X'(s)$, which gives a contradiction.
    \end{proof}
\end{prop}

\begin{rmk}
    By taking an $r$ large enough we find $\hat{\chi}_\text{CW}(X)= \hat{\chi}_\text{CW}(X')$ for $X \simeq X'$.
\end{rmk}
We can take the derivative of $\hat{\chi}^{\leq r}_\text{CW}$, to form the \textbf{weighted Euler polynomial} giving 
\begin{equation}
   \hat{\mathcal{W}}^{\leq r}_{\text{CW}}(X):= \frac{d}{dt}\big( \hat{\chi}^{\leq r}_\text{CW}(X)\big)= \sum_{a \in \text{Cells}_*(X) \text{ with } w(a)\leq r}(-1)^{|a|}\cdot w(a) \cdot t^{w(a)-1}
\end{equation}
by proposition \ref{propeul}, $\hat{\mathcal{W}}_\text{CW}^{\leq r}$  is an invariant of the filtered homotopy type of $X$. Furthermore, the value

\begin{equation}
    \mathcal{W}^{\leq r}(X):=\text{ev}_{t=1} \mathcal{W}_\text{CW}^{\leq r}(X) =\sum_{a \in \text{Cells}_*(X) \text{ with } w(a)\leq r} (-1)^{|a|}\cdot w(a)
\end{equation}

is a filtered homotopy invariant for all $r\in \R$. By taking an $r$ large enough, we obtain $\mathcal{W}(X)$ that we will refer to as the \textbf{weighted Euler characteristic} of $X$.

\begin{eg}
    Consider the filtered torus with height function as described in example \ref{torus example}. Then we find 
    \begin{equation*}
        \hat{\mathcal{W}}_\text{CW}(T)= -at^{a-1}-bt^{b-1}+ct^{c-1}
    \end{equation*}
\begin{equation*}
      \mathcal{W}(T)= c-(a+b)
\end{equation*}
\end{eg}

\begin{defn}
    Given two filtered CW-complexes with homotopy equivalent total space, we define their \textbf{matching number} as follows: First we consider $\hat{\chi}_{\text{CW}}(X)+\hat{\chi}_{\text{CW}}(X')$, then we take mod-$2$, to obtain some polynomial equivalent to something of the form $\sum_i t^{r_i}$, i.e., make every coefficient $+1$ .We then evaluate this polynomial over $\Z$ at $t=1$, and divide by two. I.e.
    \begin{equation}
        \text{Match}(X_,X'):=\frac{1}{2}\text{ev}^\Z_{t=1}\big(\hat{\chi}_{\text{CW}}(X)- \hat{\chi}_{\text{CW}}(X')\mod 2\big)
    \end{equation}
\end{defn}

\begin{rmk}
    The matching number for two filtrations on a CW complex, say $X$ and $X'$ (with equivalent total spaces), is independent of filtered homotopy equivalence. That is, if $X\simeq \bar{X}$ and $X'\simeq \bar{X}'$, then 
    \begin{equation*}
        \text{Match}(X,X') = \text{Match}(\bar{X},\bar{X}').
    \end{equation*}
This follows from $\hat{\chi}_{\text{CW}}$ being a filtered homotopy invariant.
\end{rmk}
The matching number is calculating (half) the number of `mismatched' cells in the two filtrations of the total space. We can think of this as a function $\text{Match}:\text{Filt}_h(X) \times \text{Filt}_h(X) \to \Z$ from the set of pairs of filtration on $X$ up to filtered homotopy equivalence. Note that 
\begin{equation*}
\text{ev}^\Z_{t=1}\big(\hat{\chi}_{\text{CW}}(X)+ \hat{\chi}_{\text{CW}}(X')\mod 2\big) \in 2\Z
\end{equation*}
This can be seen from $\chi(X)=\chi(X')$ hence the difference $\hat{\chi}_\text{CW}(X)-\hat{\chi}_{\text{CW}}(X')$ must have an equal number of terms with $+1$ coefficient as $-1$ coefficient, thus an even number of terms.

\begin{lemma}
    Assume $f,f':M \to \R$ are two bounded, Morse functions on a compact manifold $M$. Set, $M_f$ and $M_{f'}$ to be (some choice of) filtered CW-complexes approximating the filtered space induced from the sub-level sets of $f$ and $f'$ respectively. If $\sup_{x\in M} |f(x)-f'(x)|< \epsilon $ then 
    
    \begin{equation}\frac{|\mathcal{W}(X_f)-\mathcal{W}(X_{f'})|}{\text{Match}(X_f,X_{f'})}<\epsilon
    \end{equation}

    \begin{proof}
        First note that $\chi_\text{CW}(X_f)= \chi_{\text{CW}}(X_{f'})$ as the topology of the total space does not change. Thus we have 
        \begin{equation*}
            \text{ev}_{t=1}\big(\hat{\chi}_{\text{CW}}(X_f)\big)= \text{ev}_{t=1}\big(\hat{\chi}_{\text{CW}}(X_f)\big)
        \end{equation*}
Therefore, the difference polynomial $D(X_f,X_{f'})$ evaluated at $t=1$ gives zero. Therefore, in general $D(X_f,X_{f'})$ takes the following form 

\begin{equation*}
    D(X_f,X_{f'})=\sum_{i=1,\hdots, n} t^{a_i} - t^{b_i}
\end{equation*}
Calculating its derivative we have
\begin{equation*}
    \frac{d}{dt}D(X_f,X_{f'})= \sum_{i=1\hdots, n} a_i \cdot t^{a_i-1}- b_i \cdot t^{b_i-1} 
\end{equation*}
And evalutating at $t=1$ we obtain 
\begin{equation*}
  |\mathcal{W}(X_f)-\mathcal{W}(X_{f
'})|= |\text{ev}_{t=1}\big(  \frac{d}{dt}D(X_f,X_{f'})\big)|= |\sum_{i=1\hdots, n} a_i - b_i|\leq \sum_{i=1\hdots, n}|a_i-b_i|
\end{equation*}
We can choose to order the $a_i$s and $b_i$s such that $|a_i-b_i|< \epsilon$ for every $i$. Hence we have

\begin{equation*}
     |\mathcal{W}(X_f)-\mathcal{W}(X_{f
'})|\leq \sum_{i=1\hdots, n}|a_i-b_i| \leq n\cdot \epsilon
\end{equation*}

But we realise $n$ is exactly $\text{Match}(X_f,X_{f'})$, thus we obtain the result.

    \end{proof}

\end{lemma}

\begin{prop}\label{propwedge}
    The weighted Euler polynomial is additive with respect to $\vee$: 
    \begin{align}
     \hat{\chi}_{\text{CW}}(X \vee Y)=&  \hat{\chi}_{\text{CW}}(X) +  \hat{\chi}_{\text{CW}}(Y)
    \end{align}
    \begin{proof}
       This is an easy computation. We have $\text{Cells}_*(X \vee Y)= \text{Cells}_*(X) \cup \text{Cells}_*(Y)$ thus 
        \begin{align*}
        \hat{\chi}_{\text{CW}}(X \vee Y)&= \sum_{a\in \text{Cells}_*(X \vee Y)} (-1)^{|a|}\cdot t^{w(a)}\\
        &=\sum_{a\in \text{Cells}_*(X)} (-1)^{|a|}\cdot t^{w(a)} + \sum_{b\in \text{Cells}_*(Y)} (-1)^{|b|}\cdot t^{w(b)}\\
        &= \hat{\chi}_{\text{CW}}(X)+ \hat{\chi}_{\text{CW}}(Y).
        \end{align*}
        
    \end{proof}
\end{prop}

  \begin{rmk}
         The collection of $k$-cells in $X \times Y$ is the collection of products of $l$-cells in $X$ with $m$-cells in $Y$, with $l+m=k$. Assume $a$ is such an $l$-cell of $X$ and $b$ is a $m$-cell of $Y$, then 
        \begin{equation}
            w(a \times b)= \max\{w(a),w(b)\}
        \end{equation}
        Indeed, $a\times b$ is a $k$-cell of $(X \times Y)(r)$ iff $a$ is an $l$-cell of $X(r)$ and $b$ is an $m$-cell of $Y(r)$ hence $r\geq \max\{w(a),w(b)\}$. It follows that 
        \begin{align*}
            \hat{\chi}_{\text{CW}}(X \times Y)=& \sum_{a\times b\in \text{Cells}_*(X)\times \text{Cells}_*( Y)} (-1)^{|a \times b|}\cdot t^{\max\{w(a),w(b)\}}\\
            =&\sum_{a\in \text{Cells}_*(X)}(-1)^{|a|}\sum_{b\in \text{Cells}_*(Y)}(-1)^{|b|}\cdot t^{\max\{w(a),w(b)\}}.
        \end{align*}
    And 
        \begin{align*}
           \hat{ \mathcal{W}}(X \times Y)=&\frac{d}{dt}\sum_{a\in \text{Cells}_*(X)}(-1)^{|a|}\sum_{b\in \text{Cells}_*(Y)}(-1)^{|b|}\cdot t^{\max\{w(a),w(b)\}}\\
            =& \sum_{a\in \text{Cells}_*(X)}(-1)^{|a|}\sum_{b\in \text{Cells}_*(Y)}(-1)^{|b|}\cdot \max\{w(a),w(b)\}\cdot t^{\max\{w(a),w(b)\}-1}
            \end{align*}

    In general we can say much about the relations of $\mathcal{W}(X \times Y)$ with $\mathcal{W}(X)$ and $\mathcal{W}(Y)$. However, if we replace $\times$ with $\bar{\times}$ we have the following results.
        \end{rmk}

\begin{prop}
    $\hat{\chi}_{\text{CW}}(X \bar{\times} Y)=\hat{\chi}_{\text{CW}}(X) \cdot \hat{\chi}_{\text{CW}}(Y)$.
    \begin{proof}
         Again the collection of $k$-cells in $X \bar{\times} Y$ is the collection of $a \times b$-cells where $a$ is an $l$-cell, $b$ is an $m$-cell and $k=l+m$. Now however the weight of such cells is given by
       \begin{equation}
           w(a\times b)=w(a)+w(b)
       \end{equation}

      Cells of the form $* \times b$ and $a \times *$ are eternal and thus do not contribute to the count. Thus we obtain

      \begin{align*}
          \hat{\chi}_{\text{CW}}(X \bar{\times} Y) = &\sum_{a\in \text{Cells}_*(X)}\sum_{b\in \text{Cells}_*(Y)}(-1)^{|a|+|b|}\cdot t^{w(a)+w(b)}\\
          =& \sum_{a\in \text{Cells}_*(X)}\sum_{b\in \text{Cells}_*(Y)}(-1)^{|b|}\cdot t^{w(b)} \cdot (-1)^{|a|} \cdot t^{w(a)}\\
          =&\big( \sum_{a\in \text{Cells}_*(X)} (-1)^{|a|} \cdot t^{w(a)}\big) \cdot \big(\sum_{b\in \text{Cells}_*(Y)}(-1)^{|b|}\cdot t^{w(b)} \big)\\
         =& \hat{\chi}_{\text{CW}}(X) \cdot \hat{\chi}_{\text{CW}}(Y)
      \end{align*}

    \end{proof}
\end{prop}
  
\begin{cor}
    The weighted Euler characteristic and classical Euler characteristic satisfy:
    \begin{equation}
        \mathcal{W}(X \bar{\times} Y) =\mathcal{W}(X)\cdot \chi(Y) + \mathcal{W}(Y) \cdot \chi(X).
    \end{equation}
    \begin{proof}
      Using the previous proposition we have:
       \begin{align*}
           \hat{\mathcal{W}}(X \bar{\times} Y)= &\frac{d}{dt} \hat{\chi}_{\text{CW}}(X \bar{\times} Y)= \frac{d}{dt}\big(\hat{\chi}_{\text{CW}}(X) \cdot \hat{\chi}_{\text{CW}}(Y)\big)\\
           =& \frac{d}{dt}\hat{\chi}_{\text{CW}}(X) \cdot \hat{\chi}_{\text{CW}}(Y)+ \hat{\chi}_{\text{CW}}(X) \cdot  \frac{d}{dt}\hat{\chi}_{\text{CW}}(Y)\\
           =& \hat{\mathcal{W}}_{\text{CW}}(X) \cdot \hat{\chi}_{\text{CW}}(Y)+ \cdot \hat{\chi}_{\text{CW}}(X)\cdot \hat{\mathcal{W}}_{\text{CW}}(Y).
       \end{align*}
    \end{proof}
\end{cor}

\begin{rmk}
    The above result implies that   $\mathcal{W}(X \bar{\times} Y) =\mathcal{W}(X)\cdot \chi(Y) + \mathcal{W}(Y) \cdot \chi(X)$.
\end{rmk}
\begin{cor}\label{corsmash}

 We have the following identities:
    \begin{align}
        \hat{\chi}_{\text{CW}}(X \bar{\Smash} Y)= &\hat{\chi}_{\text{CW}}(X) \cdot \hat{\chi}_{\text{CW}}(Y)\\
      \hat{\mathcal{W}}_{\text{CW}}(X\bar{\Smash}Y)  =& \hat{\chi}_{\text{CW}}(X) \cdot \hat{\mathcal{W}}_{\text{CW}}(Y)+ \hat{\chi}_{\text{CW}}(Y)\cdot \hat{\mathcal{W}}_{\text{CW}}(X)
    \end{align}
    
\begin{proof}
    By definition $X\bar{\Smash}Y(r)= (X\bar{\times} Y)(r)/X \vee Y$, thus we have $\text{Cells}_*(X \Smash Y)= \text{Cells}_*(X \times Y) \setminus \text{Cells}(X \vee Y)$. But note that all cells in $X \vee Y$ are eternal, and so do not contribute to the count. Hence we obtain the same count as for $\bar{\times}$. 

\end{proof}
    
\end{cor}
\begin{rmk}

Finally we remark that we have equalities:
\begin{align}
     \chi(X \bar{\Smash} Y)= &\chi(X) \cdot \chi(Y)\\
      \mathcal{W}(X\bar{\Smash}Y)  =& \chi(X) \cdot \mathcal{W}(Y)+ \chi(Y)\cdot \mathcal{W}(X).
\end{align}
    
\end{rmk}
\subsection{A filtered loop-suspension adjunction}
Given two filtered spaces $X$ and $Y$, we can form a new filtered space $\text{Map}(X,Y)$ given by letting 
\begin{equation}\text{Map}(X,Y)( r):=\{f:X \to Y | \lceil f \rceil =r\}\end{equation}
be the subspace $\text{Map}(X,Y)\in \text{Top}_*$, of maps with shift $r$. Notice this is indeed a filtered topological space, as any morphism of shift $r$ can be viewed as a morphism of shift $s>r$ via inclusion. In particular, we have inclusions

\begin{equation}
    i_{r,s}^{\text{Map}(X,Y)}:\text{Map}(X,Y)(r) \to \text{Map}(X,Y)(s)
\end{equation}
given by pushing $f$ forward under the filtration maps of $Y$. Note this defines a filtered topological space in the sense of definition \ref{filttopdef}.
\begin{defn}
We define the \textbf{loop space} $\Omega X$ of a filtered space $X$ to be 
\begin{equation}\Omega(X):=\text{Map}(S^1_0,X)\end{equation}
where $S^1_0$ is $S^1$ with the zero filtration.
\end{defn}

\begin{rmk}
A map of shift $r$, from $S^1_0$ with the zero filtration, into any filtered space is uniquely determined by $f(0):S^1(0) \to X(r)$ as the following must commute
\begin{equation*}
    \begin{tikzcd}
    & X(s)\\
        S^1 \ar[ru,"f(s-r)"]\ar[r,swap,"f(0)"]& X(r)\ar[u,"i_{r,s}"]\\
    \end{tikzcd}
\end{equation*}
hence 
\begin{equation}\Omega(X)( r)=\text{Map}(S^1, X(r))= \Omega(X(r)).\end{equation}
\end{rmk}
\begin{prop}
There is an adjunction 
\begin{equation}\text{Map}(\S X,Y)(r) \simeq \text{Map}(X,\Omega Y)( r)\end{equation}

\begin{proof}
An element of $\text{Map}(\S X,Y)( r)$ is a family of maps
\begin{equation*}f(t):(\S X)( t)\to Y( t+r)\end{equation*}
which by definition of filtered suspension is a family of maps
\begin{equation*}
f(t): \S (X( t)) \to Y( t+r)\end{equation*}
by the usual loop-suspension adjunction we have that this is equivalent to a collection of maps
\begin{equation*}
\bar{f}(t): X(t) \to \Omega(Y( t+r))\end{equation*}
and so a collection of maps
\begin{equation*}\bar{f}( t): X( t) \to (\Omega Y)(t+r).\end{equation*}
Furthermore, one finds 

\begin{equation*}
    \begin{tikzcd}
        X(s)\ar[r,"\bar{f}(s)"] & (\Omega Y)(s+r) \\
        X(t)\ar[r,"\bar{f}(t)"]\ar[u,"i_{t,s}^X"] & (\Omega Y)(r+t)\ar[u,swap,"i_{t+r,s+r}^{\Omega Y}"]
    \end{tikzcd}
\end{equation*}
commutes for all $s\geq t$. Indeed, recall that the usual loop-suspension adjunction $\text{Hom}(\S X,Y) \simeq \text{Hom}(X, \Omega Y)$ is given by $\phi: (f:\S X \to Y) \mapsto (\phi(f): X \to \text{Map}(S^1,Y))$ with $(\phi(f)(x))(\tau):= f(x,\tau)$, where we identify $\S X$ with $S^1_0 \bar{\Smash} X$. We therefore find
\begin{align*}
   ( i_{t+r,s+r}^{\Omega Y} \circ \bar{f}(t))(x)(\tau)=&(i_{t+r,s+r}^{\Omega Y} \circ \phi(f(t)))(x)(\tau)\\
   =&i^Y_{t+r,s+r}\circ \phi(f(t))(x)(\tau)\\
   =& i^Y_{t+r,s+r} \circ (f(t))(x,\tau)\\
   =&(f(s))(x,\tau) \circ i^X_{t,s}\\
   =& \phi(f(s))(x)(\tau) \circ i^X_{t,s}\\
   =& \bar{f}(s) \circ i_{t,s}^X.
\end{align*}
Hence $\bar{f}$ defines a filtered morphism, and so we have said adjunction.
\end{proof}
\end{prop}

\section{Filtered stable homotopy}

In general homotopy classes of maps $[X,Y]$ between two topological spaces need not form a group. Though, if $Y$ is an H-group or if $X$ is a co-H-group then there exists a natural group structure. In particular, if $X\simeq \S X'$ is a suspension space or if $Y \simeq \Omega Y'$ is a loop space then $[X,Y]$ carries a natural group structure, as supsensions are co-H-groups and loop spaces are H-groups. In the setting of filtered spaces, it follows that $[X,Y]$ carries a persistence group structure (persistence module with each level forming a group rather than module) if $X\simeq \S X'$ is a filtered suspension space or if $Y \simeq \Omega Y'$ is a filtered loop space. Moreover, $[X,Y]$ can be seen to be a filtered Abelian group, i.e., persistence $\Z$-module  if $X\simeq \S^2 X'$ or $Y\simeq \Omega^2 Y'$. Thus if we want to study filtered spaces via the techniques of persistence categories, then we should pass to stable homotopy. 

There is a morphism of filtered CW complexes (and general filtered spaces) we call the \textbf{filtered pinch map}:

\begin{equation}
    \S X \to \S X \vee \S X
\end{equation}
simply given levelwise by the usual pinch map:

\begin{equation*}
    \S X \simeq \text{Cone}\big( X \xrightarrow{i_1} \text{Cone}(X\xrightarrow{i_0} \text{Cyl}(X))) \xrightarrow{\text{pinch}} \S X \vee \S X
\end{equation*}
with $i_{(-)}$ being the inclusion of $X$ into $\text{Cyl}(X)$ along $\{(-)\}\times X$, and $\text{pinch}$ being the projection to the quotient of $\text{Cyl}(X)$ by identifying $\{\frac{1}{2}\}\times X$ with the basepoint. In the filtered setting this map can be seen to have zero shift.

Given two filtered morphisms of equal shift $f,f': \S X \to Y$ we can form their \textbf{filtered concatenation}: 
\begin{equation}
    f+ f': \S X \to Y
\end{equation}
with shift $\lceil f+f'\rceil=\lceil f\rceil = \lceil f' \rceil$. It is given by the composition:
\begin{equation}
    f+f' := \S X \xrightarrow{\text{pinch}}\S X \vee \S X \xrightarrow{f\vee f'} Y.
\end{equation}
i.e., the levelwise concatenation. It follows from the levelwise co-H-group structure on $\S X$ that $[\S X,Y](r):=\text{Hom}_{\ff \text{Top}_*}(\S X,Y)(r)/\sim$, the $r$-th level of the morphism of filtered spaces up to filtered homotopy, is a group. And moreover, since $\S^2 X$ has an Abelian co-H-group structure, $[\S^2 X,Y](r)$ is an Abelian group. Notice that there exist group morphisms 

\begin{equation*}
    (i_{r,s})_*:[\S ^2X ,Y](r) \to [\S^2 X,Y](s)
\end{equation*}
induced by the filtered morphism maps. Thus we can view $[\S^2X,Y]$ as a persistence module (in $\text{Mod}_\Z$). In this set up we naturally obtain a \text{filtered Freudenthal suspension theorem}: For large enough $n\in \Z$ we have an isomorphism of (Abelian) groups:

\begin{equation}
    \S: [\S^n X ,\S^n Y](r) \to [\S^{n+1}X, \S^{n+1}Y](r) .
\end{equation}

If we stabilise our filtered topological category via $\S^\infty$ we will therefore obtain a persistence category. In the remainder of this section we consider a persistence version of the Spanier-Whitehead category of CW-complexes, given by stabilising the category of filtered CW complexes. We show that it is a triangulated persistence category, and explore what the fragmentation metrics defined on TPCs are calculating.
\subsection{A persistence Spanier-Whitehead category}

Given two filtered CW complexes $X$ and $Y$ we can define Abelian groups, 

\begin{equation}
    \pp\text{SW}(X,Y)(r):=\text{Colim}_{n\geq 0}[\S^n X , \S^n Y](r).
\end{equation}

for every $r\in \R$. Given $r<s$ we obtain induced group morphisms 

\begin{equation}
    (i_{r,s}^{SW})*:\pp\text{SW}(X,Y)(r) \to \pp\text{SW}(X,Y)(s)
\end{equation}

More generally, we define:

\begin{equation}
    \pp\text{SW}((X,k)(Y,l))(r):=\text{Colim}_{n>k,l}[\S^{n+k}X,\S^{n+l}Y](r)
\end{equation}
and obtain a persistence module in Abelian groups $\{\pp\text{SW}((X,k),(Y,l)),i_{r,s}^{\pp\text{SW}}\}$.

We define the \textbf{persistence Spanier-Whitehead category} $\pp \text{SW}$, to be the category with objects pairs $(X,k)$ with $X$ a filtered CW complex and $k\in \Z$. The hom-sets are persistence modules given by $\pp\text{SW}((X,k),(Y,l))$. This category comes with natural shift functors for all $a\in \R$, $\ss^a:\pp \text{SW} \to \pp \text{SW}$ given by

\begin{equation}
    \ss^a(X,k)=(\ss^a X,k)
\end{equation}
recalling that $(\ss^a X)(r)=X(r-a)$. There are also functors $[n]:\pp \text{SW} \to \pp \text{SW}$ given by
\begin{equation}
    [n](X,k):= (X,k+n).
\end{equation}

There is a canonical inclusion: $\iota:\text{Ho}\ff \text{Top}_*^{CW} \to \pp \text{SW}$, given by

\begin{equation*}
    \iota(X) = (X,0)
\end{equation*}
and with 
\begin{equation*}
    \iota(f:X \to Y):=\text{Colim}_n(\S^n f):(X,0) \to (Y,0).
\end{equation*}
Note that this restricts to an inclusion $\iota_0:(\text{Ho}\ff \text{Top}_*^{CW})_0 \to (\pp \text{SW})_0$
\begin{theorem}
The category $\pp \text{SW}$ is a triangulated persistence category with shift functors $\ss$ and triangulated functor $[1]$.
\end{theorem}

\begin{proof}
   The usual Spanier-Whitehead category is triangulated, with exact trianlges being triangles

   \begin{equation}
       \begin{tikzcd}
           (X,k)\ar[r,"f"] & (Y,l)\ar[r,"g"] & (Z,m)\ar[r,"h"] & [1](X,k)
       \end{tikzcd}
   \end{equation}
   which, up to some even suspension, are isomorphic to the image of a cofiber sequence under $\iota$, i.e., for some $n\in 2\Z$ there is an isomorphism of triangles
    \begin{equation}
       \begin{tikzcd}
           {[n]}(X,k)\ar[r,"{[n]f}"]\ar[d,"\cong"] & {[n]}(Y,l)\ar[r,"{[n]g}"]\ar[d,"\cong"]  & {[n]}(Z,m)\ar[r,"{[n]h}"] \ar[d,"\cong"] & {[n+1]}(X,k)\ar[d,"\cong"] \\
           (X',0)\ar[r,"\iota_0(f')"] & (Y',0)\ar[r,"\iota_0(i)"] & (\text{Cone}(f'),0)\ar[r,"\iota_0(j)"] & {[1]}(X',0)
       \end{tikzcd}
   \end{equation}

   We prove the statement in two parts:
   \begin{enumerate}
       \item[Part 1:] $(\pp \text{SW})_0$ is triangulated.  

    This will follow directly from the usual Spanier-Whitehead category being triangulated. In fact there is a general construction for `the Spanier-Whitehead category of $(\mathcal{H},\S)$' for a pointed model category $\mathcal{H}$ with suspension $\S$ (see \cite{De}) which is always triangulated. The category $(\text{Top}_*^{\pp SW})_0$ can be seen to be precisely the Spanier-Whitehead category of $((\ff\text{Top}_*)_0,\S)$. Indeed, by choice of filtration on mapping cones and suspensions being taken levelwise, we see that requirements to define the Spanier-Whitehead category of $((\ff \text{Top}_*)_0,\S)$ are satisfied, and hence it is a well defined triangulated category. Note furthermore $[1]$ commutes with $\ss^a$ for all $a\in \R$ by simply realising $\ss^a[1](X,k)=(\ss^a X,k+1)$

    \item[Part 2:] $\eta_r^{(X,k)}:(X,k) \to \ss^{-r}(X,k)$ has $r$-acyclic cone.
    By definition 
    \begin{equation*}
    [(X,k),\ss^{-r}(X,k)](0):= \text{Colim}[\S^{n+k}X,\S^{n+k}\ss^{-r}X](0)
    \end{equation*}
    We realise that the map $\eta_r^{(X,k)}$ is induced by $\eta_r^{\S^kX}:\S^kX \to \ss^{-r} \S^kX $ in $\ff \text{Top}_*$. Which levelwise is given by the natural inclusions
    \begin{equation*}
        \eta_{r}^{\S^k X}(t):\S^k X(t) \to \S^k \ss^{-r}X(t)=\S^k X(t+r)
    \end{equation*}
    The cone of this map in $\ff\text{Top}_*$ is, up to (zero) homotopy equivalence, given by 
    \begin{equation*}
        \text{Cone}(\eta_r^{\S^k X})(t)= \S^{k}X(t+r)/\S^{k}X(t).
    \end{equation*}
We find that the maps 
\begin{equation*}
    \eta_r^{\text{Cone}(\eta_r^{\S^k X})}(s):\text{Cone}(\eta_r^{\S^k X})(s) \to \ss^{-r}\text{Cone}(\eta_r^{\S^k X})(s)=\text{Cone}(\eta_r^{\S^k X})(s+r)
\end{equation*}
are simply given by the inclusions

\begin{equation*}
    \S^{k}X(s+r)/\S^{k}X(s) \to \S^{k}X(s+2r)/\S^{k}X(s+r)
\end{equation*}
which one sees are naturally nullhomotopic. In particular, after passing back to stable homotopy one finds that 
\begin{equation*}
    \eta_r^{\text{Cone}(\eta_{r}^{(X,k)})}=0.
\end{equation*}
i.e., $\text{Cone}(\eta_r^{(X,k)})$ is $r$-acyclic.
\end{enumerate}

\end{proof}

\begin{lemma}
    There is an equivalence of triangulated categories $\Phi: \text{SW} \to \pp \text{SW}_\infty$, with $ \text{SW}$ being the usual Spanier-Whitehead category of (finite) CW-complexes.
    \begin{proof}
        In the limit category $\pp \text{SW}_\infty$ any two filtered CW complexes with total space become isomorphic. Indeed, by assumption $-\infty<\lfloor X \rfloor \leq \lceil X \rceil < +\infty$. Assume $X$ and $X'$ are filtered spaces with $X(\infty)=X'(\infty)$, then in $\pp SW$ we can find an isomorphism $f: (X,n) \to (X,n)$ given by $[n] \iota(f': X \to  X')$, with $f'\in \text{Hom}_{\ff \text{Top}_*}(X,X')(\lceil X'\rceil - \lfloor X \rfloor)$ being induced by the equality $X(\lceil X'\rceil)= X'(\lceil X' \rceil) $. $[n]f$ has shift $\lceil X'\rceil - \lfloor X \rfloor$ but in the limit $\pp\text{SW}_\infty$ gives an isomorphism $(X,n) \cong (X',n)$. We can therefore define a functor $\Phi: \text{SW} \to \pp \text{SW}_\infty$, given by
        \begin{equation*}
            \Phi((X,n))(r)=\begin{cases}
            (X,n) & r\geq 0\\
            (*,n) & r<0
            \end{cases}
        \end{equation*}
        Clearly this is essentially surjective. If we denote $(X,n)$ to be such that $(X,n)(r)=(X,n)$ for $r\geq 0$ and trivial for $r<0$, and similarly for $(X',n')$, then in $\pp\text{SW}$ we have $\text{Hom}_{\pp \text{SW}}((X,n),(X',n'))(0)=\text{Hom}_{\pp \text{SW}}((X,n),(X',n'))(r)=\text{Hom}_{\text{SW}}((X,n),(X',n'))$ for all $r\geq 0$. Thus we have full faithfullness.
    \end{proof}
    
\end{lemma}

\subsection{Strict exact triangles}\label{strictexactsub}
Before proceeding to discuss the associated (pseudo) metrics induced by this TPC structure it is worth understanding what form strict exact triangles take in $\pp \text{SW}_0$. Recall a triangle in $\pp \text{SW}_0$ 
 \begin{equation}
       \begin{tikzcd}
           (X,k)\ar[r,"f"] & (Y,l)\ar[r,"g"] & (Z,m)\ar[r,"h"] & [1](X,k)
       \end{tikzcd}
   \end{equation}
is exact if there exists an $n\in 2\Z$ and a zero-isomorphism of triangles

 \begin{equation}\label{isooftri}
       \begin{tikzcd}
           {[n]}(X,k)\ar[r,"{[n]f}"]\ar[d,"\cong"] & {[n]}(Y,l)\ar[r,"{[n]g}"]\ar[d,"\cong"]  & {[n]}(Z,m)\ar[r,"{[n]h}"] \ar[d,"\cong"] & {[n+1]}(X,k)\ar[d,"\cong"] \\
           (X',0)\ar[r,"\iota_0(f')"] & (Y',0)\ar[r,"\iota_0(i)"] & (\text{Cone}(f'),0)\ar[r,"\iota_0(j)"] & {[1]}(X',0)
       \end{tikzcd}
   \end{equation}

where the lower triangle is the image of a cofibre sequence in $(\ff \text{Top}_*)_0$. In $\pp \text{SW}_0$ a triangle 
\begin{equation}\label{triang}
       \begin{tikzcd}
           (X,k)\ar[r,"f"] & (Y,l)\ar[r,"g"] & (Z,m)\ar[r,"h"] & \ss^{-r}[1](X,k)
       \end{tikzcd}
   \end{equation}

is strict exact of weight $r$ if there exists a commutative diagram

\begin{equation}
       \begin{tikzcd}
          & & \ss^r(Z',m')\ar[d,"\psi"]\ar[dr,"\ss^rh"]\\
          (X,k)\ar[r,"f"] & (Y,l)\ar[r,"\bar{g}"]\ar[rd,"g"] & (Z',m')\ar[d,"\phi"]\ar[r,"\bar{h}"] &{[1]}(X,k)\\
           & & (Z,m)\ar[r,"h"] & \ss^{-r}[1](X,k)
       \end{tikzcd}
   \end{equation}

with central line being an exact triangle and $\phi$ an $r$-isomorphism (with $\phi \circ \psi =\eta_r$). Assuming the central triangle is exact via the isomorphism in \ref{isooftri}, then consider

\begin{equation}
    \begin{tikzcd}
    (X',0)\ar[d,"\cong"] \ar[r,"\iota_0(f')"]&\ar[d,"\cong"] (Y',0)\ar[rr,"\iota_0(i)"] && \ar[d,"\cong"](\text{Cone}(f'),0) \ar[r,"\iota_0(j)"]& {[1]}(X',0)\ar[d,"\cong"]\\
        {[n]}(X,k) \ar[r,"{[n]f}"] & {[n]}(Y,l)\ar[drr,"{[n](\phi \circ \bar{g})}"] \ar[rr,"{[n] \bar{g}}"]&& {[n]}(Z',m')\ar[d,"{[n]\phi}"]\ar[r,"{[n]\bar{h}}"] & {[n+1]}(X,k)\\
        & && {[n]}(Z,m) \ar[r,"{[n]h}"]& \ss^{-r}{[n+1]}(X,k)
    \end{tikzcd}
\end{equation}
Thus, the triangle \ref{triang}, being strict exact of weight $r$ implies that there exists an $n\in 2\Z$ and a commutative diagram

\begin{equation}\label{stricttri}
    \begin{tikzcd}
        & && \ss^r[n](Z,m)\ar[d,"\vartheta"]\ar[dr,"\ss^r\varepsilon"] &\\
        (X',0) \ar[r,"\iota_0(f')"]&\ar[drr,swap,"\varphi \circ \iota_0(i)"] (Y',0)\ar[rr,"\iota_0(i)"] && (\text{Cone}(f'),0) \ar[d,"\varphi"]\ar[r,"\iota_0(j)"]& {[1]}(X',0)\\
       {} &{} && {[n]}(Z,m)\ar[r,"\varepsilon"] & \ss^{-r}{[1]}(X',0)
    \end{tikzcd}
\end{equation}
with $\varphi$ being an $r$-isomorphism (and $\varphi \circ \vartheta =\eta_r$). 

\begin{prop}
  If $\alpha:(A,m) \simeq_r (B,l)$ is an $r$-isomorphism, then there exists some $A'$ and $B'\in \ff\text{Top}_*$, and an $n\in 2\Z$ such that there is a $\iota_0(\beta):(A',0) \simeq_r (B',0)$ and the following commutes
    \begin{equation}
        \begin{tikzcd}
            {[n]}(A,m)\ar[r,"\cong"] \ar[d,"{[n]\alpha}"]& (A',0)\ar[d,"\iota_0(\beta)"] \\
            {[n]}(B,l)\ar[r,"\cong"] & (B',0)
        \end{tikzcd}
    \end{equation}
    \begin{proof}
        By definition of $r$-isomorphism there must exist some $\alpha: (A,m) \to (B,l)$ such that there is an exact triangle

        \begin{equation*}
            \begin{tikzcd}
                (A,m) \ar[r,"\alpha"]& (B,l) \ar[r]& (K,k)\ar[r] & {[1]}(A,m)
            \end{tikzcd}
        \end{equation*}
        with $\eta_r^{(K,k)}=0$. Thus there is some $n\in 2\Z$ and an isomorphism of triangles
         \begin{equation*}
            \begin{tikzcd}
               {[n]} (A,m) \ar[d,"\cong"]\ar[r,"{[n]\alpha}"]& {[n]}(B,l) \ar[d,"\cong"]\ar[r]& {[n]}(K,k)\ar[r]\ar[d,"\cong"] & {[n+1]}(A,m)\ar[d,"\cong"]\\
                (A',0)\ar[r,"\iota_0(\beta)"]  & (B',0)\ar[r,"\iota_0(i)"] & (K',0)\ar[r,"\iota_0(j)"] & [1](A',0) 
            \end{tikzcd}
        \end{equation*}
        Thus the leftmost square is the desired commutative square. It is easy to show that $\iota_0(\beta)$ is an $r$-isomorphism. Since $(K,k)$ is $r$-acyclic $[n](K,k)$ must also be $r$-acyclic. Thus $(K',0)$ is zero isomorphic to an $r$-acyclic object. Denoting this isomorphism by $f:[n](K,k) \to (K',0)$ we find 
        \begin{equation*}
            \eta_r^{(K',0)}= f \circ f^{-1} \circ \eta_r^{(K',0)}= f \circ \eta_r^{[n](K,k)} \circ \ss^r f^{-1}=0
        \end{equation*}
        meaning that $(K',0)$ is $r$-acyclic and thus $\iota_0(\beta)$ is an $r$-isomorphism.
    \end{proof}
\end{prop}

The preceding result implies that we can replace (after possibly choosing larger $n$) the diagram in \ref{stricttri} with a (zero) isomorphic diagram of the form

\begin{equation}\label{stricttri2}
    \begin{tikzcd}
        & && \ss^r(C,0)\ar[d,"\vartheta"]\ar[dr,"\ss^r\varepsilon"] &\\
        (X',0) \ar[r,"\iota_0(f')"]&\ar[drr,swap,"\varphi \circ \iota_0(i)"] (Y',0)\ar[rr,"\iota_0(i)"] && (\text{Cone}(f'),0) \ar[d,"\varphi"]\ar[r,"\iota_0(j)"]& {[1]}(X',0)\\
       {} &{} && (C,0)\ar[r,"\varepsilon"] & \ss^{-r}{[1]}(X',0)
    \end{tikzcd}
\end{equation}

In particular if a triangle in $\pp \text{SW}_0$ is strict exact of weight $r$, then after suitably many even suspensions it is zero isomorphic to a strict exact triangle of weight $r$ which is the image of a cofiber sequence in $(\ff \text{Top}_*)_0$. The converse can be easily be verified.

\subsection{Fragmentation distances and sizes}
The TPC structure on $\pp \text{SW}$ allows us to construct fragmentation (pseudo) metrics on the class of objects. Furthermore these metrics can be pulled back via the functor $i:\ff \text{Top}_* \to \pp \text{SW}$ to metrics on $\ff \text{Top}_*$. Let $\ff$ be a family of objects in $\ff\text{Top}_*$ and let $X,Y$ also be objects of $\ff \text{Top}_*$, then we define the fragmentation distance on $\ff\text{Top}_*$

\begin{equation}
d_{\ff\text{Top}_*}^\ff(X,Y):=\inf \big\{d^{\bar{\ff}}_{{\pp\text{SW}}}((\bar{X},0),(\bar{Y},0)): \bar{X} \simeq X, \bar{Y} \simeq Y\big\}
\end{equation}
Here, $d_{\pp\text{SW}}^{(-)}$ is the associated fragmentation metric induced by the TPC structure on $\pp \text{SW}$, $\bar{X}$ is a filtered CW approximation of $X$, and $\bar{\ff}:=\{(\bar{F},0):F\in \ff\}$. The choice of family $\ff$ is key to the definition of these metrics, indeed two different choices and induce completely different metrics. Thus we wish to choose a `nice' family which allows for distances that are comparable to previously defined distances. Recall that we define the filtered $k$ sphere at level $0$ by $S^k$, it is given by
\begin{equation*}
    S^k(t):=\begin{cases}
        * & t<0\\
        S^k & t\geq 0
    \end{cases}
\end{equation*}
We define the family $
\mathbb{S}:=\{ S^k: k\in \N\}$ and set $d_\mathbb{S}:=d^{\mathbb{S}}_{\ff \text{Top}_*}$. The idea here is that this family will allow one to generate any CW complex via attaching cones. The `size' of these attaching cones induces the distances we wish to study. These fragmentation metrics also induce a notion of size of a filtered topological space, we define the size of a filtered space $X$ to be:

\begin{equation*}
    |X|_{\mathbb{S}}:=d_{\mathbb{S}}(*,X).
\end{equation*}

\begin{prop}
    The size $|X|_{\mathbb{S}}$ is finite. 
    \begin{proof}
    By assumption, each filtration step $X(r)$ is equivalent to a finite CW-complex. Furthermore, the set $\text{Spec}(X)$ is also finite and the filtration stablises. Thus we can construct $X$ out of finitely many attaching cones. In particular, the image of each of these cone sequences in $\pp \text{SW}$ give exact triangles $\Delta_i$ each with finite weight. Hence we can bound $|X|_{\mathbb{S}}$ by $\sum_i^n w(\Delta_i)<\infty$ .

    \end{proof}
\end{prop}
\begin{prop}
    The distance $d_{\mathbb{S}}(X,Y)$ is finite. 
    \begin{proof}
        This will follow from $|X|_\mathbb{S}$ and $|Y|_{\mathbb{S}}$ both being finite. Indeed, we can attach enough cones to $X$ to such that the resulting space is contractible. We then attach cones to to build $Y$. We see that $d_\mathbb{S}(X,Y)\leq |X|_\mathbb{S}+|Y|_\mathbb{S}$.
    \end{proof}
\end{prop}

\begin{lemma}\label{weightsumcrit}
    Let $X$ be a compact connected smooth manifold, and $f:X \to \R$ be non-negative, bounded, Morse. Let $X_f$ be the filtered space corresponding to the sub-level sets of $f$, then \begin{equation}|X_f|_\mathbb{S}\leq \sum_{x\in\text{Crit}(f)} f(x).\end{equation}
    \begin{proof}
        The homotopy type of $X$ will only change as we pass over critical points of $f$, it is a well known result of Morse theory that, if $x\in \text{Crit}(f)$ is such that the submanifold $f^{-1}[f(x)-\epsilon,f(x)+\epsilon]$ contains only $x$ as a critical point, then $f^{-1}\{\leq f(x)+\epsilon\}$ is homotopy equivalent to $f^{-1}\{\leq f(x)-\epsilon\}$ with a cell attached. The dimension of this cell depends on the index of $x$ as a critical point. More generally, if there are many (isolated) critical points all with the same critical value, then $f^{-1}\{\leq f(x)+\epsilon\}$ is homotopy equivalent to $f^{-1}\{\leq f(x)- \epsilon\}$ with a cell attached for each critical point. We can view these cell attachments as mapping cone sequences (in $\text{Top}_*$)
        \begin{equation*}
            S^k \xrightarrow{\phi} X(f(x)-\epsilon) \to X(f(x)+\epsilon) \to S^{k+1}.
        \end{equation*}
        Note that $\text{Cone}(S^k)\cong D^{k+1}$ are homeomorphic. Thus we build a filtered CW-approximation $\bar{X}_f$ to $X_f$ as follows: Start with $\bar{X}_f^0=*$, at the first critical value of $f$ (assume this to be $r_1\in \R$) we need to attach cells of some dimension, say $k_1$. We need to attach a cell for every $x\in \text{Crit}(f)$ such that $f(x)=r_1$. Assume there are $n_1$ such cells and label them $\{x_1^{r_1},\hdots, x_{n_1}^{r_1}\}$. We iteratively attach each cell to the basepoint $*$, via cofiber sequences
        \begin{align*}
          &  S^{k-1} \xrightarrow{\phi_1^{r_1}} * \to S^k \to S^k\\
       &   \vdots\\
         & S^{k-1} \xrightarrow{ \phi_{n_1}^{r_1}} (\bar{X}_f^{0})_{n_1-1} \to  (\bar{X}_f^{0})_{n_1} \to S^k  
        \end{align*}
         Where $(\bar{X}_f^{0})_{a}$ denotes the constructed space after $a$ cell attachments. The shift of each map $\phi^{r_1}_a$ is $r_1$ and so 
         \begin{equation*}
             (\bar{X}_f^{0})_{a}(r)\simeq\begin{cases}
                 (\bar{X}_f^0)_a & r\geq r_1\\
                 * & r<r_1
             \end{cases}
         \end{equation*}
         (note the abuse of notation identifying the space constructed with the filtered space constructed). Each of these sequences induce strict exact triangles of weight $r_1$ in $\pp \text{SW}_0$ given by

         \begin{equation*}
             (S^{k-1},0) \xrightarrow{\iota\big((\eta_{0,-r_1})\circ \phi_a^{r_1}\big)} \ss^{-r_1}\big((\bar{X}_f)^0_a,0\big) \to \ss^{-r_1}\big((\bar{X}_f)^0_{a+1},0\big) \to \ss^{-r_1}[1](S^k,0).
         \end{equation*}
         Indeed notice that at filtration level $r$ we have triangles
          \begin{equation*}
             (S^{k-1}(r),0) \xrightarrow{\iota\big((\eta_{0,-r_1})\circ \phi_a^{r_1}\big)} \big((\bar{X}_f)^0_a(r+r_1),0\big) \to \big((\bar{X}_f)^0_{a+1}(r+r_1),0\big) \to [1](S^k(r+r_1),0).
         \end{equation*}
         Hence we have $|(\bar{X}_f^0)_{n_1}|_\mathbb{S}\leq n_1 \cdot r_1$. We then iterate this process over each critical value of $f$ to obtain filtered CW-complexes $(\bar{X}_f)^{r_m}_{n_m}$ with 
         
         \begin{equation*}|(\bar{X}_f)^{r_m}_{n_m}|_\mathbb{S}\leq n_1 \cdot r_1 + \hdots + n_m \cdot r_m=\sum_{x \in \text{Crit}(f):f(x)\leq r_m} f(x)
         \end{equation*}
         Since we assume $f$ is bounded, then we obtain
         \begin{equation*}
             |X_f|_\mathbb{S}\leq |\bar{X}_f|_\mathbb{S}= \sum_{x \in \text{Crit}(f)} f(x)
         \end{equation*}
    \end{proof}

\end{lemma}
\begin{rmk}
    If we allow for attachings of multiple cells at once, i.e., we consider a family $\mathbb{S}_\vee=\{\bigvee_{i=1,\hdots,n}S^k_i : k\in \N\}$ and its associated pseudometric. Then we find that 
    \begin{equation}
        |X_f|_{\mathbb{S}_\vee}\leq \sum_{r\in \text{CritVal}(f)} r
    \end{equation}
    where CritVal denotes the set of critical values of $f$. That is, $r\in \text{CritVal}(f)$ if $r=f(x)$ for some $x\in \text{Crit}(f)$. This follows from realising that we can attach all cells at each critical value at once. Hence, we have only one exact triangle for each critical value.
\end{rmk}

Given two $C^0$-close Morse functions on a manifold $f,f':M\to \R$, they need not be close with respect to the fragmentation metric $d_\mathbb{S}$. Indeed, $d_\mathbb{S}$ counts `the weighted sum of critical points', so if one were to perturb $f$ to $f'$ adding in many more critical points though keeping $f'$ $C^0$-close, then the distance via $d_\mathbb{S}$ can be arbitrarily large. This can be seen in the following example.

\begin{eg}
    Let $S^2_h:=(S^2,h)$ be the filtered two sphere with filtration given by the height function, $h:S^2 \to \R$ with critical points at $h(N)=1$ and $h(S)=0$ where $N,S$ are the north and south poles respectively. This filtered space can be seen to be weakly equivalent to 
    \begin{equation*}
        S^2_h(r)\simeq\begin{cases}
            * & r<1\\
            S^2 & r\geq 1
        \end{cases}
    \end{equation*}
    Note $|S^2_h|=1$.
    Denote by $S^2_{h'}:=(S^2,h')$ where $h'$ is the height function on the two sphere which has critical points at $N_1$, $N_2$ $M$ and $S$ with $h(N_1)=h(N_2)=1$, $h(M)=\frac{1}{2}$ and $h(S)=0$. We see that this filtered space is given (up to weak equivalence) by
    \begin{equation*}
        S^2_{h'}(r)\simeq\begin{cases}
            * & r< \frac{1}{2}\\
            S^1 & \frac{1}{2}\leq r < 1\\
            S^2 & r\geq 1
        \end{cases}
    \end{equation*}
    and note $|S^2_{h'}|=\frac{5}{2}$. Firstly we consider persistence homology of the two filtered spaces, we find

    \begin{equation*}
        H_*(S^2_h)(r)=\begin{cases}
            \R\langle a_0 \rangle & r<1\\
            \R \langle a_0 \rangle\oplus \R \langle a_2 \rangle& r\geq 1
        \end{cases} \hspace{20pt} H_*(S^2_{h'})(r)=\begin{cases}
            \R\langle b_0\rangle & r< \frac{1}{2}\\
            \R \langle b_0 \rangle \oplus \R \langle b_1 \rangle & \frac{1}{2}\leq r < 1\\
            \R \langle b_0 \rangle \oplus \R \langle b_2 \rangle & r\geq 1
        \end{cases}
    \end{equation*}
    The corresponding barcodes are given by:
    
    And the interleaving distance between these two persistence modules can be calculated to be $d_{\text{int}}( H_*(S^2_h), H_*(S^2_{h'}))=\frac{1}{4}$.  Now we wish to calculate the distance between these two filtered spaces using $d_{\mathbb{S}}$. We will build $S^2_{h'}$ out of $S^2_h$ and $\mathbb{S}$. Firstly, consider the map $\varphi:S^0 \to S^2_h$ with shift $\frac{1}{2}$ given by $ x\mapsto *$. One finds that the mapping cone is given by (up to homotopy)
    \begin{equation*}
        \text{Cone}(\varphi)(r)=\begin{cases}
            * & r< \frac{1}{2}\\
            S^1 & \frac{1}{2} \leq r <1\\
            S^1 \vee S^2 & r\geq 1
        \end{cases}
    \end{equation*}
    Next we look to the map $\psi:S^1 \to \text{Cone}(\varphi)$ of shift $1$, where $\psi(r)$ is the trivial map for $r<0$ and the inclusion into the copy of $S^1$ for $r\geq 0$. The map will `fill in' the copy of $S^1$, so that we have

    \begin{equation*}
        \text{Cone}(\psi)(r)= \begin{cases}
            * & r< \frac{1}{2}\\
            S^1 & \frac{1}{2}\leq r < 1\\
            S^2 & r\geq 1
        \end{cases}
    \end{equation*}

Consider in $\pp \text{SW}_\infty$ exact triangles, 
\begin{equation*}
        \begin{tikzcd}
           (S^0,0)\ar[r,"{[\iota(\varphi)]}"] & (S^2_h,0)\ar[r] & (\text{Cone}({[\iota(\varphi)]}),0)\ar[r] & {[1]}(S^0,0) \\
           (S^1,0)\ar[r,"{[\iota(\psi)]}"] & (\text{Cone}({[\iota(\varphi)]}),0) \ar[r]& (S^2_{h'},0)\ar[r] & {[1]}(S^1,0)
        \end{tikzcd}
    \end{equation*}
    One can calculate the sum of their weights to be be (less than or equal to) $\frac{3}{2}=1+\frac{1}{2}$.
\end{eg}

An interesting aspect of working with stable homotopy means that spaces that are not homotopy equivalent can be stably equivalent. A straight forward example of this is the torus $T=S^1 \times S^1$ and the space given by $S^1 \vee S^1 \vee S^2$, this follows from $\S(X \times Y) \simeq \S X \vee \S Y \vee \S(X \Smash Y)$ and $\S(X \vee Y) \simeq \S X \vee \S Y$. In the setting of filtered spaces this means that some filtered spaces $X,Y$ that are not filtered homotopy equivalent, can have $d_\ss(X,Y)=0$.

\begin{eg}
 If we take trivial filtrations on $T:=S^1 \times S^1$ and $T':=S^1 \vee S^1 \vee S^2$ spaces so that
    \begin{equation*}
        T(r)=\begin{cases}
            * & r<0\\
            T & r\geq 0
        \end{cases} \hspace{20pt} T'(r)=\begin{cases}
            * & r<0 \\
            T' & r\geq 0
        \end{cases}
    \end{equation*}
  Then there is an exact triangle in $\pp \text{SW}_0$ given by
  \begin{equation*}
      0 \to (T,0) \xrightarrow{\cong} (T',0) \to 0
  \end{equation*}
  with the isomorphism given by levelwise isomorphisms between $\S T(r) $ and $\S T'(r)$. Hence, we find $d_\mathbb{S}(T,T')=0$.
\end{eg}

Recall the definition of linearisation via strict exact triangles in the zero level category of a TPC, an object $X\in \cc$ has linearisation $L=(F_1,\hdots,F_n)$ if $X$ is an iterated weighted cone over the objects $F_i$, i.e., there is a sequence of strict exact triangles $\cc_0$

\begin{equation*}\begin{tikzcd}[ampersand replacement =\&]
    \Delta_1 : \& F_1\ar[r] \& 0 \ar[r] \& X_1\ar[r]  \& \ss^{-r_1}TF_1\\
    \Delta_2 : \& F_2 \ar[r] \& X_1 \ar[r] \& X_2 \ar[r] \& \ss^{-r_2}TF_2\\
    \vdots \& \& \vdots \& \vdots\\
    \Delta_n: \& F_n\ar[r]  \& X_{n-1} \ar[r] \& X' \ar[r] \& \ss^{-r_n}TF_n                    
\end{tikzcd}\end{equation*}
Here, $X'$ is zero-isomorphic to $X$. Denote by $\text{Lin}^\ff(X)$ the set of linearisations of $X$ with $F_i\in \ff$. The weight of such a linearisation we will denote by $w(L)$ and is given by the sum of the weights of the triangles, i.e., $w(L)=r_1 +\hdots + r_n$.
\begin{defn}
 Define a function $\lambda_X^\ff:\text{Lin}^\ff(X) \to \Lambda$ given by 

\begin{equation}
    \lambda_X^\ff(L):=\sum_{F_i\in L} t^{r_i}.
\end{equation}
\end{defn}

\begin{rmk}
    We can evaluate $\lambda^\ff_X$ at values of $t$ via a function $ev_{t=\alpha}:\Lambda \to \R$. Note that $\text{ev}_{t=1}(\lambda^\ff_X(L))$ is simply the number of elements of the linearisation $L$. Thus,
    \begin{equation}
        \inf\big\{ \text{ev}_{t=1}(\lambda_X^\ff(L)): L\in \text{Lin}^\ff(X)\big\}
    \end{equation}
    gives a count of the smallest amount of iterated cones needed to construct $X$ from the family $\ff$.
\end{rmk}

We call $L=(F_1,\hdots,F_n)$ a linearisation of $X\in \ff\text{Top}_*$ if $\iota(L):=(\iota(F_1), \hdots, \iota(F_n))$ is a linearisation of $\iota(X)$ in $\pp\text{SW}$, and similarly write $\text{Lin}^\ff(X)$ as the set of linearisations of $X$ with $F_i\in \ff$. We will consider the functions $\lambda_X^\mathbb{S}:\text{Lin}^\mathbb{S}(X) \to \Lambda$.
\begin{prop}
    Let $f:M \to \R$ be a non-negative, bounded Morse function, then there exists a linearisation $L\in \text{Lin}^\mathbb{S}(X_f)$ with 
    
    \begin{equation}
   \text{ev}_{t=1}( \lambda_{X_f}^\mathbb{S}(L))=\#\text{Crit}(f)
    \end{equation}
    \begin{proof}
        This linearisation is exactly the one given in the proof of Lemma \ref{weightsumcrit}, we are attaching one cell for each critical point.
    \end{proof}
\end{prop}

Consider also the derivative of $\lambda^\ff_X$ given by

\begin{equation}
    \frac{d}{dt}\lambda_X^\ff(L)= \sum_{F_i\in L}r_i\cdot t^{r_i-1}
\end{equation}
Evaluation at $t=0$ then gives
\begin{equation}
    \text{ev}_{t=0}\big( \frac{d}{dt}\lambda_X^\ff(L)\big)=\sum_{F_i\in \ff} r_i= w(L).
\end{equation}

Thus we find that, for the linearisation given in Lemma \ref{weightsumcrit}, we have:

\begin{prop}
Let $f:M \to \R$ be a non-negative, bounded Morse function, then there exists a linearisation $L\in \text{Lin}^\mathbb{S}(X_f)$ with 
\begin{equation}
     \text{ev}_{t=0}\big( \frac{d}{dt}\lambda_{X_f}^\mathbb{S}(L)\big)=\sum_{x\in \text{Crit(f)}}f(x)
\end{equation}
\end{prop}
\begin{rmk}
   We note that we can rewrite the size of an object in a TPC via linearisation in $\ff$ by:
    \begin{equation}
    |X|_\ff:=d^\ff(0,X)= \inf\big\{ \text{ev}_{t=1}\big( \frac{d}{dt}\lambda_X^\ff(L)\big): L\in \text{Lin}^\ff(X)\big\}.
    \end{equation}
\end{rmk}

We can also define a `weighted version' of the Euler characteristic.

\begin{defn}
    Given a filtered space $X\in \ff \text{Top}_*$, with CW-approximation $\bar{X}$, and an $L\in \text{Lin}^\mathbb{S}(\bar{X})$ we define the \textbf{Euler Polynomial of $X$ relative $L$}:
    \begin{equation}
        \chi^\mathbb{S}_{X}(L):= \sum_{S^{k_i}\in L}(-1)^{k_i+1}\cdot  t^{r_i} 
    \end{equation}
\end{defn}

Notice that
\begin{align*}
    \text{ev}_{t=1}\big( \chi_{X}^\mathbb{S}(L)\big)=&\sum_{S^{k_i}\in L}(-1)^{k_i+1}\\
    =&\sum_k (-1)^{k+1} \cdot \#(S^{k}\in L)
\end{align*}
where $\#(S^k\in L)$ denotes the number of times $S^k$ appears in the linearisation $L$. We also have the derivative:
\begin{equation}
   \mathcal{W}^\mathbb{S}_X(L):= \frac{d}{dt}\big(\chi_{X}^\mathbb{S}(L)\big)=\sum_{S^{k_i}\in L}(-1)^{k_i+1}\cdot r_i\cdot t^{r_i-1}
\end{equation}
which evaluated at $t=1$, gives

\begin{align}\label{eulweight}
     \text{ev}_{t=1}\big( \mathcal{W}^\mathbb{S}_X(L)\big)=&\sum_{S^{k_i}\in L} (-1)^{k_i+1}\cdot r_i.
\end{align}

\begin{lemma}
$\chi_X^\mathbb{S}$ recovers $\hat{\chi}_{\text{CW}}(X)$, thus is independent of linearisation.

\begin{proof}
    A general $L\in \text{Lin}^\mathbb{S}(X)$ is a sequence of strict exact triangles

    \begin{equation*}
      \begin{tikzcd}[ampersand replacement =\&]
    \Delta_1 : \& (S^{k_1},0)\ar[r,"f_1"] \& (*,0) \ar[r] \& \ss^{-r_1}[1](S^{k_1},0) \ar[r]  \& \ss^{-r_1}[1](S^{k_1},0)\\
    \Delta_2 : \& (S^{k_2},0) \ar[r,"f_2"] \& \ss^{-r_1}[1](S^{k_1},0) \ar[r] \& (X_2,n_2) \ar[r] \& \ss^{-r_2}[1](S^{k_2},0)\\
    \vdots \& \& \vdots \& \vdots\\
    \Delta_m: \& (S^{k_m},0)\ar[r,"f_n"]  \& (X_{m-1},n_{m-1}), \ar[r] \& (X_m,n_m) \ar[r] \& \ss^{-r_n}[1](\ss^{k_m},0)                    
\end{tikzcd}
    \end{equation*}

After suitably large even suspension, each of these strict exact triangles are zero isomorphic to strict exact triangles in the image of $\iota_0$. Thus we can replace this sequence with a sequence of triangles with each $n_i=0$, and with $X'_m\simeq X$, i.e., we replace the linearisation $L$ with some $L'$:

 \begin{equation*}
      \begin{tikzcd}[ampersand replacement =\&]
   (S^{k_1+2l_1},0)\ar[r,"\iota_0(f'_1)"] \& (*,0) \ar[r] \& \ss^{-r_1}[1](S^{k_1+2l_1},0) \ar[r]  \& \ss^{-r_1}[1](S^{k_1+2l_1},0)\\
 (S^{k_2+2l_2},0) \ar[r,"\iota_0(f'_2)"] \& \ss^{-r_1}[1](S^{k_1+2l_1},0) \ar[r] \& (X'_2,0) \ar[r] \& \ss^{-r_2}[1](S^{k_2+2l_2},0)\\
  \& \vdots \& \vdots\\
  (S^{k_m+2l_m},0)\ar[r,"\iota_0(f'_m)"]  \& (X'_{m-1},0), \ar[r] \& (X'_m,0) \ar[r] \& \ss^{-r_n}[1](\ss^{k_m+2l_m},0)                    
\end{tikzcd}
    \end{equation*}

As we are suspending an even number of times, the value of 
\begin{equation*}
    (-1)^{k_i+1}=(-1)^{k_i+1+2l}
\end{equation*}
furthermore because we are replacing with zero isomorphic strict exact triangles the value of each $r_i$ will remain the same. Hence, the value of $\chi_X^\mathbb{S}(L')=\chi_X^\mathbb{S}(L)$. Now we notice that 
\begin{equation*}
    \chi_X^\mathbb{S}(L')= \hat{\chi}_\text{CW}(X')
\end{equation*}
But this is simply $\hat{\chi}_{\text{CW}}(X)$ as $X \simeq X'$. 
\end{proof}
\end{lemma}

\begin{rmk}
 Evaluation at $t=1$ of $\chi_X^\mathbb{S}$ recovers the Euler characteristic of the total space $X$:
    \begin{equation}
\text{ev}_{t=1}\big(\chi_{X}^\mathbb{S}\big)=\chi(X)
    \end{equation}
   
\end{rmk}

\begin{rmk}
    We can also define \begin{equation}
 (\chi^\mathbb{S}_X)^{\leq r}(L):=      \sum_{S^{k_i}\in L \text{ with }r_i\leq r}(-1)^{k_i+1}t^{r_i} 
    \end{equation}

    which will be independent of $L$ and recover $\hat{\chi}^{\leq r}_{\text{CW}}(X)$, and similarly for $(\mathcal{W}^\mathbb{S}_X)^{\leq r}(L)$.
\end{rmk}

\begin{cor}
    The derivative $\mathcal{W_X^\mathbb{S}}$ recovers $\hat{\mathcal{W}}(X)$, and is also independent of $L$.
\end{cor}

\begin{rmk}
    One could choose to work with the limit category $\pp \text{SW}_\infty$ and define a Euler polynomial and weighted Euler polynomial. In this setting, we must replace the criteria of the terminal object $(X_m,0)$ in the sequence of exact triangles to not be isomorphic to the object whose polynomials we calculate $X$, but must be actually equal to $X$. This is due to isomorphisms in $\pp \text{SW}_\infty$ not necessarily being images of zero-isomorphisms in $\pp \text{SW}_0$. 
\end{rmk}

\subsection{K-group}
The Euler characteristic of (pointed) CW complexes induces an isomorphism of rings $\chi: K(\text{SW}) \to \Z$, this is due to $S^0$ generating $\text{SW}$ as a triangulated category and realising that if $[(X,m)]\in K(\text{SW})$ then $[(X,m)]=(-1)^m(\chi(X)-1)[(S^0,0)]$. We briefly explore how this argument extends to a ring isomorphism $K(\pp \text{SW}_0)\to \Lambda$. 

The $K$-group of $\cc_0$ for some TPC $\cc$, has a natural $\Lambda_P$-module structure, given by $[\ss^{r}A]=t^r\cdot [A]$ We show that $\bar{\Smash}$ extends to $\pp \text{SW}$ and descends to a $\Lambda_P$ product structure on $K(\pp \text{SW}_0)$. More precisely we show

\begin{lemma}
    In $K(\pp \text{SW}_0)$ the smash product, defined by 
    \begin{align}
       (X,n) \bar{\Smash} (Y,m):= (X\bar{\Smash} Y,n+ m)&
    \end{align}
    induces a product
    \begin{equation}[(X,n)]\cdot [(Y,m)]:=[(X,n)\bar{\Smash}(Y,m)]
    \end{equation}
     making $K(\pp \text{SW}_0)$ an $\Lambda_P$-algebra.
\begin{proof}
    To prove this statement, we use a result of \cite{BCZ2} (Lemma 3.3.5 and show that $\bar{\Smash}$ defines a TPC tensor structure on $\pp \text{SW}$ (see definition 3.3.1 of \cite{BCZ2}). We show that $\bar{\Smash}$ satisfies each of the properties:

    \begin{itemize}
        \item[(i)] Take $(A,a)\in \pp\text{SW}$, if $f\in \text{Hom}_{\pp \text{SW}}\big((X,n), (Y,m)\big)(r) $, then 
        
        \begin{equation*}1_{(A,a)}  \bar{\Smash}(i_{r,s}(f))=i_{r,s}(1_{(A,a)} \bar{\Smash}f)\in \text{Hom}_{\pp \text{SW}}\big( (A \bar{\Smash} X,a+n), (A\bar{\Smash}Y,a+m) \big)(r)
        \end{equation*}
        Indeed, the filtrations in $\pp \text{SW}$ is induced by the filtrations in $\ff \text{Top}^{\text{CW}}_*$. Thus this follows by remark \ref{sizeofsmashmorphism}.

    \item[(ii)] $A\bar{\Smash} (-)$ restricts to a triangulated functor on $\pp \text{SW}_0$. By definition, a triangle in $\pp \text{SW}_0$ is triangulated if after some even number of suspensions it is isomorphic to the image of a cofiber sequence from $\ff \text{Top}_*^{\text{CW}}$. Thus, it is enough to check that $\bar{\Smash}$ in $\ff \text{Top}_*^{\text{CW}}$ sends cofiber sequences to cofiber sequences. Consider a sequence 
    \begin{equation*}
        \begin{tikzcd}
            X\ar[r,"f"] & Y\ar[r]& \text{Cone}(f) \ar[r]&\S X
        \end{tikzcd}
    \end{equation*}
    Applying $A \bar{\Smash}(-)$ we have
      \begin{equation*}
        \begin{tikzcd}
           A \bar{\Smash} X\ar[r,"  1_A \bar{\Smash}f"] &   A \bar{\Smash}Y\ar[r]&   A \bar{\Smash}\text{Cone}(f) \ar[r]&  A \bar{\Smash}\S X.
        \end{tikzcd}
    \end{equation*}

    We must show that $\text{Cone}(1_A \bar{\Smash} f)\simeq A \bar{\Smash}\text{Cone}(f)$. By definition, we have $\text{Cone}(1_  A \bar{\Smash} f)(t)=\text{Cone}\big( ( 1_  A \bar{\Smash}f)(t)\big)$. We then find
    \begin{align*}
        \text{Cone}\big((1_A \bar{\Smash} f)(t)\big)=& \bigcup_{s+l=t} \text{Cone}\big( 1_A(s) \Smash f(l)\big)\\
        =& \bigcup_{s+l=t}\text{Cone}(1_{A(s)}\Smash f(l)\big)\\
        =& \bigcup_{s+l=t}A(s) \Smash \text{Cone}(f(l))\\
        =& A \bar{\Smash} \text{Cone}(f).
    \end{align*}
Furthermore, $\bar{\Smash}$ can be seen to be additive on $\pp \text{SW}_0$ by distributivity over $\vee$.
    \item[(iii)] Finally, we check that $[1]((A,a)\bar{\Smash} (X,n)) \simeq [1] (A,a) \bar{\Smash} (X,n) \simeq (A,a) \bar{\Smash} [1]( X,n)$. This is equivalent to 
    \begin{equation*}
        \big(\S (A\bar{\Smash} X),a+n\big)\simeq \big(\S A \bar{\Smash}X, a+n\big)\simeq \big(A \bar{\Smash}\S X,a+n\big).
    \end{equation*}
    Thus again, it is enough to check that in $\ff \text{Top}_*^{\text{CW}}$ we have
    \begin{equation*}
        \S (A\bar{\Smash} X)\simeq \S A \bar{\Smash}X\simeq A \bar{\Smash}\S X
    \end{equation*}
    But this is clear, as $\S(-)= S^1_0 \bar{\Smash}(-)$.
    \end{itemize}

\end{proof}
    
\end{lemma}

\begin{theorem}\label{ktheorylem}
    The weighted Euler polynomial $\hat{\chi}_{\text{CW}}$, induces an isomorphism of $\Lambda_P$-algebras $\mathcal{X}: K(\pp \text{SW}_0) \to \Lambda_P$, with 

    \begin{equation}
        \mathcal{X}: [(X,n)]\mapsto (-1)^n \cdot \hat{\chi}_{\text{CW}}(X).
    \end{equation}

\end{theorem}

This result is a persistence version of a known result in stable homotopy theory:

\begin{lemma}[Theorem 4.5.6 \cite{Mo}]
    There is a ring isomorphism $\chi:K(\text{SW}) \cong \Z $ given by 
    \begin{equation*}
    \chi([(X,n)])=(-1)^n\cdot \chi(X).
    \end{equation*}
\end{lemma}
Here, $\chi(X)$ is the reduced Euler characteristic. The proof comes from noticing that $[(X,n)]=(-1)^n\cdot \chi(X) \cdot [(S^0,0)]$. The proof of Theorem \ref{ktheorylem} follows in a similar manner:

\begin{proof}[Proof of Theorem \ref{ktheorylem}]
    Firstly, we claim that $[(X,n)]=(-1)^n \cdot \hat{\chi}_{\text{CW}}(X) \cdot [(S^0_0,0)]$. Let $S^k_t$ denote the filtered CW-complex given by 
    \begin{equation*}
        S^k_l(r):=\begin{cases}
            S^k & r\geq l\\
            * & r< l
        \end{cases}
    \end{equation*}
   We find that:

    \begin{align*}
        [(S^k_l,0)]=[\ss^l( S^k_0,0)]=&t^l \cdot [(S^k_0,0)]\\
        =&(-1)^k \cdot t^l \cdot [(S^0_0,0)]
    \end{align*}
    where the last equality follows from the exact triangle

    \begin{equation*}
        (S^{k-1}_0,0) \to * \to (S^k_0,0) \to (S^k_0,0)
    \end{equation*}
  giving $[(S^{k-1}_0,0)]=-[(S^k_0,0)]$, i.e., in general $[(X,n)]=-[(X,n+1)]$ in $K(\pp \text{SW}_0)$. Any filtered CW complex is formed by attaching of spheres $S^k_l$ via cofiber sequences in $(\ff \text{Top}_*)_0$, thus we have a sequence of exact triangles in $\pp \text{SW}_0$ of the form

  \begin{equation*}
      \begin{tikzcd}
          \bigvee_{i\in I_0} (S^0_{l_i},0) \ar[r]& (X_0,0)\ar[r] & (X_1,0)\ar[r] &  \bigvee_{i\in I_0} (S^0_{l_i},1)\\
            \bigvee_{i\in I_1} (S^1_{l_i},0)\ar[r] & (X_1,0) \ar[r]& (X_2,0)\ar[r] & \bigvee_{i\in I_1} (S^1_{l_i},1)\\
            \vdots & & & \\
              \bigvee_{i\in I_{n-1}} (S^{n-1}_{l_i},0) \ar[r]& (X_{n-1},0)\ar[r] & (X,0)\ar[r] &  \bigvee_{i\in I_{n-1}} (S^{n-1}_{l_i},1)
      \end{tikzcd}
  \end{equation*}
  Where $X_0$ is the `filtered zero-skeleton'. Implying that in $K(\pp \text{SW}_0)$ we have 
  \begin{align*}
      [(X,0)]=& [(X_{n-1},0)]-\big( (-1)^{n-1}\cdot [(S^0_0,0)]\sum_{i\in I_{n-1}} t^{l_i}\big)\\
      =&\big([(X_{n-2},0)]-\big( (-1)^{n-2}\cdot [(S^0_0,0)]\sum_{i\in I_{n-2}} t^{l_i}\big) - \big( (-1)^{n-1}\cdot [(S^0_0,0)]\sum_{i\in I_{n-1}} t^{l_i}\big)\\
      =& \hdots\\
      =& \sum_j \sum_{i \in I_j}(-1)^j \cdot t^{l_i} \cdot [(S^0_0,0)]\\
      =&\hat{\chi}_{\text{CW}}(X) \cdot [(S^0,0)].
  \end{align*}

  Finally, as $[(X,n)]=-[(X,n+1)]$, we obtain the general form:
  \begin{equation*}
      [(X,n)]=(-1)^n \cdot \hat{\chi}_{\text{CW}}(X) \cdot [(S^0_0,0)].
  \end{equation*}

From the results of proposition \ref{propwedge} and corollary \ref{corsmash} we see that $\mathcal{X}$ induces a ring morphism. Furthermore, as $\hat{\chi}_{\text{CW}}(\ss^l X)=t^l \cdot \hat{\chi}_{\text{CW}}(X)$, (something that is easy to verify) we find that $\hat{\chi}$ induces a morphism of $\Lambda_P$-algebras. Moreover, this is an isomorphism as $K(\pp \text{SW}_0)$ as a $\Lambda_P$-algebra is one dimensional with generator $[(S^0_0,0)]$ and $\mathcal{X}[(S^0_0,0)]=1$, generates $\Lambda_P$.
\end{proof}

\subsection{Filtered Spectra}
The Spanier-Whitehead category can be seen to be a subcategory of the stable homotopy category. There are various equivalent (in homotopy) constructions of this category, the most simple of which is given by sequential spectra. The constructions of this paper can be seen to extend to this setting. One can define a category of filtered spectra similarly to that of filtered topological spaces. i.e., a \textbf{filtered spectrum} is a spectrum $(X,\sigma)$ equipped with a filtration $\ff$ consisting of an $\R$-indexed family of subspectra.
\begin{equation}
\ff=\{X(r): X(r)\xhookrightarrow{i_r} X\}_{r\in \R} \end{equation}
where $i_r$ is the canonical inclusion map.
This family should satisfy:
\begin{enumerate}
\item The $n$-th level space of $X$, should be a filtered space with $X(r)_n=X_n(r)$.
    \item For all $r\in \R_{}$ for all $n\in \N$, we have $i_r(*_{(X(r))_n})=*_{X_n}$, i.e., they have a common basepoint.
    \item  For all $s<r$ we have $X(s) \subset X(r)$, we denote the inclusion map $i_{s,r}$.
    \item There exists an $r_0\in \R$ such that for all $r< r_0$ we have $X(r)=X'$, i.e., the filtration is stabilises below to some spectrum $X'$.

    \item There exists an $r_1\in \R$ such that for all $r>r_1$ we have $X(r)=X$, i.e., the filtration stabilises above to $X$ which we refer to as the \text{total spectrum}.
    \item The homotopy type of $X(r)$ can only change at finitely many $r$, we call this set $\text{Spec}(X)$.

\end{enumerate}
We can define morphisms of filtered spectra to be an $\R$-indexed family of morphisms of spectra that commute with the filtration maps. Most of the constructions in the category of filtered spaces will pass to constructions on the category of filtered spectra, $\ff \text{Spectra}$. In particular, we can restrict our attention to filtered CW-spectra, via an analogue of the CW approximation for spectra. 

\begin{lemma}
    Any filtered spectrum $X$ is weakly equivalent to a filtered CW-spectrum $\bar{X} \to X$, where the following homotopy commutes for all $r\leq s$
    \begin{equation}
        \begin{tikzcd}
            X(s)\ar[r,"\alpha(s)"] & \bar{X}(s)\\
            X(r)\ar[r,"\alpha(r)"]\ar[u,"i_{r,s}"] & \bar{X}(r)\ar[u,"i_{r,s}"]
        \end{tikzcd}
    \end{equation}
    \begin{proof}
        The proof is almost identical to the proof of Lemma \ref{CWspace} using the CW-appoximation theorem for spectra (see \cite{EKM} Theorem 1.5).
    \end{proof}
\end{lemma}

In order to pass to a persistence category by taking taking the homotopy category $\text{Ho}\ff \text{Spectra}^{\text{CW}}$, we need a filtered version of spectrification. In the category of filtered topological spaces we can define a filtered free infinite loop space functor $Q: \ff \text{Top}_* \to \ff \text{Top}_*$ given by 
\begin{equation}
    Q(X)(r):= \text{Colim} \Omega^k \S^k X(r)
\end{equation}

There are filtration maps $i_{r,s}:Q(X)(r) \to Q(X)(s)$ induced by the maps $i_{r,s}:X(r) \to X(s)$. The functor extends to a functor $Q: \ff \text{Spectra} \to \ff \text{Spectra}$ by 

\begin{equation*}
    Q(X)(r)_n=Q(X_n(r)).
\end{equation*}
We can then pass to a filtered stable homotopy category by defining hom-sets as the homotopy classes of maps
\begin{equation}
    \text{Hom}_{\text{Ho}\ff \text{Spectra}^\text{CW}}(X,Y)(r)=\text{Hom}_{\ff \text{Spectra}^{\text{CW}}}(Q(X),Q(Y))(r)/\sim
\end{equation}
Note that this is a persistence category with shift functors $\ss^a:\text{Ho}\ff \text{Spectra}^{\text{CW}}\to \text{Ho}\ff \text{Spectra}^{\text{CW}}$ given by $\ss^a(X)(r)=X(r-a)$.

\begin{lemma}
    The filtered stable homotopy category $\text{Ho}\ff \text{Spectra}^{\text{CW}}$ is a TPC. 
    \begin{proof}
        Exact triangles in $(\text{Ho}\ff \text{Spectra}^{\text{CW}})_0$ will be families of exact triangles 
        \begin{equation*}\{X(r) \to Y(r) \to Z(r) \to \S X(r)\}_{r\in \R}\end{equation*} 
        which commute with the filtration maps on $X,Y$ and $Z$. It follows from the usual stable homotopy category being triangulated that this is triangulated. The maps $\eta_r^X$ fit into exact triangles of the form 

        \begin{equation*}
            X \xrightarrow{\eta_r^X}\ss^{-r}X \to K \to \S X
        \end{equation*}
        with $K$-acyclic. Indeed, we have that 
        \begin{equation*}
            X(t) \xrightarrow{\eta_r^X(t)}\ss^{-r}X(t)
        \end{equation*}
        is given by the inclusion $X(t) \hookrightarrow X(t+r)$, thus has cone

        \begin{equation*}
            \text{Cone}(X(t) \hookrightarrow X(t+r))_n\simeq X_{n}(t+r)/X_n(t) 
        \end{equation*}
        Naturally, this object is $r$-acyclic as 
        \begin{equation*}
            \eta_{r}^{X_n(t+r)/X_n(t)}: X_n (t + r) / X_n (t) \to X_n (t + 2r) / X_n (t + r) \end{equation*} is null homotopic for all $r\in \R$ and $n\in\Z$.
    \end{proof}
\end{lemma}

One reason for working with $\text{Ho}\text{Spectra}$ and not with $\text{SW}$ is due to Brown's representation theorem; every generalised cohomology theorem $E^*$ can be represented by a spectrum $E$ with

\begin{equation*}
    E^n(-)\cong[-,E_n]
\end{equation*}

dually, every generalised homology theorem $E_*$ can be represented by a spectrum $E$ with
\begin{equation*}
    E_n(-)\cong\pi_n(E_n\Smash X).
\end{equation*}

It makes sense that a similar statement should hold for `persistence homology theories'. For example, consider the usual persistence (singular) homology $\text{H}_*(X;f)$ of some space $X$ with functional $f:X \to \R$. Let $K(\Q)$ be the Eilenberg-MacLane spectrum for the rationals, give this spectrum the zero filtration, i.e., $K(\Q)(r)=K(Q)$ for all $r\geq 0$ and is the trivial spectrum for $r<0$. 
Then \begin{align*}
    (K(Q)_n \bar{\Smash} X_f)(r)=&\bigcup_{s+t=r} K(\Q)_n(s) \Smash X_f(t)\\
    =& \bigcup_{t\leq r} K(\Q)_n \Smash X_f(t)\\
    =&K(\Q)_n \Smash X_f(r)
\end{align*}
Giving
\begin{equation}
    \text{H}_*^{\leq r}(X;\Q))=\pi_n \big( K(\Q)_n\Smash X_f(r)\big)= \pi_n\big((K(\Q)_n \bar{\Smash} X_f)(r)\big)=\pi_n^r \big( K(\Q)_n \bar{\Smash}X_f\big).
\end{equation}

Note that this is a specific example of what one could call a \textbf{generalised filtered homology theory}. Such an object should consist of the following:

\begin{enumerate}
    \item A functor $E^\leq_*:\ff \text{Top}^{\text{CW}}_* \to \text{Ab}_\Z^\pp$ from filtered CW complexes to $\Z$ graded persistence modules in Abelian groups. i.e., $\text{Ab}_\Z^\pp:= [(\R,\leq), \text{Ab}_\Z]$.
    \item Shift zero isomorphisms $s: E^\leq_n(-) \to E^\leq_{n+1}(\S(-))$.
\end{enumerate}
such that:
\begin{itemize}
\item $E^\leq_*$ restricts to maps $E^\leq_*(-):\text{Hom}_{\ff \text{Top}_*^{\text{CW}}}(X,Y)(r) \to \text{Hom}_{\text{Ab}_\Z^\pp}(E^\leq_*(X),E^\leq_*(Y))(r) $ such that $E^\leq_*\circ i_{r,s}=i_{r,s}\circ E^\leq_*$.
\item Isomorphisms $E^\leq_*(\ss^r(-)) \to \ss^{r}E^\leq_*(-)$.
    \item If $f\sim f'\in \text{Hom}_{\ff \text{Top}_*^{\text{CW}}}(X,Y)(r)$ are homotopic maps, then \begin{equation*}
        E^\leq_*(f)=E^\leq_*(f')\in \text{Hom}_{\text{Ab}_\Z^\pp}(E^\leq_*(X),E^\leq_*(Y))(r).
    \end{equation*}
    \item Let $i\in \text{Hom}_{\ff \text{Top}_*^\text{CW}}(A,X)(0)$ be levelwise an inclusion, then \begin{equation*}
        E^\leq_*(A) \xrightarrow{E^\leq_*(i)} E^\leq_*(X) \xrightarrow{E^\leq_*(j)}E^\leq_*( \text{Cone}(i))
    \end{equation*}
    is exact in $(\text{Ab}_\Z^\pp)_0$. 
    \item $E^\leq_*(\text{Cone}(\eta_r^X))$ is $r$-acylic.
\end{itemize}
It is easy to verify that persistence homology satisfies these axioms.
\begin{rmk}
This is a first attempt at what a generalised filtered homology theory should be. Notice that similar to the example of $H_*^\leq(X;\Q)$ if one takes any spectrum and gives it the zero filtration, then $E_n^\leq (-):=\pi_n^r(E\bar{\Smash}(-))$ defines a generalised filtered homology theory. In general though one could take any filtration on a spectrum and define such a theory. This follows from $S^1_0\bar{\Smash}E\bar{\Smash}X \simeq E \bar{\Smash}S^1_0 \bar{\Smash X}$, and shifting commuting with the filtered smash product:
\begin{equation}
    E^{\leq}_n(\S(-))=\pi_n^r(E \bar{\Smash} \S(-))\cong\pi_n^r\big(\S \big(E\bar{\Smash}(-)\big)\big)= \pi_{n+1}^r(E\bar{\Smash}(-))=E_{n+1}^{\leq}(-)
\end{equation}
and
\begin{equation}
    E^{\leq}_n(\ss^a(-))=\pi_n^r\big(E\bar{\Smash}(\ss^a(-))\big)=\pi_n^r\big( \ss^a \big( E \bar{\Smash}(-)\big)\big)=\pi_n^{r-a}\big((E\bar{\Smash}(-)\big)=\ss^aE^\leq_n(-).
\end{equation}

\end{rmk}

To conclude, we remark that this filtration construction (i.e., constructing $\ff \text{Top}_*$ and $\ff \text{Spectra}$ from $\text{Top}_*$ and $\text{Spectra}$) should be a part of a general construction on model categories. One would expect that if the model category is stable then the filtered homotopy category should return a TPC.
\newpage

\phantomsection


\addcontentsline{toc}{section}{References}
\begin{thebibliography}{ams}
\bibitem[Ar]{Ar}
M. Arkowitz, \textit{Introduction to Homotopy Theory} Springer New York, NY,  (2011)

\bibitem[BCZ1]{BCZ1}
 P. Biran, O. Cornea, J. Zhang, \textit{Triangulation, Persistence, and Fukaya categories}, (2023), 
arXiv.2304.01785 

\bibitem[BCZ2]{BCZ2}
    P. Biran, O. Cornea, J. Zhang, \textit{Persistence K-theory} (2024) arXiv:2305.01370 . 



\bibitem[BK]{BK}
A. Bondal, M. Kapranov, \textit{Enhanced Triangulated Categories}. Math. USSR Sb. 70 93 (1991) 


 
 \bibitem[BM]{BM}
P. Bubenik, N. Milićević, \textit{Homological algebra for persistence modules}. Foundations of Com-
putational Mathematics,  (2019), pages 1–46

\bibitem[Bo]{Bo}
F. Borceux, \textit{Handbook of Categorical Algebra: Basic category theory}. Cambridge University Press  (1994) 

\bibitem[BS]{BS}
P. Bubenik, J. A. Scott, \textit{Categorification of persistent homology}  Discrete Comput Geom 51, (2014) 600–627 


\bibitem[Ca]{Ca}
G. Carlsson, \textit{Topology and data}. Bull. Amer. Math. Soc. (N.S.), 46(2): (2009)  255–308.


\bibitem[De]{De}
I. Dell’Ambrogio, \textit{The Spanier-Whitehead category is always
triangulated} Master’s thesis, ETH Zurich, (2003) 

\bibitem[Di]{Di}
T. Dieck, \textit{Algebraic Topology}. European Mathematical Society, Zurich. (2010)

\bibitem[DG]{DG}
P. Dłotko, D.Gurnari, \textit{Euler Characteristic Curves and Profiles: a stable shape invariant for big data problems} (2023) arXiv:2212.01666



\bibitem[DP]{DP}
 A. Dold, D. Puppe, \textit{Duality, trace and transfer}, Topology, A collection of papers, Trudy Mat. Inst. Steklov., 154, (1983), 81–97
\bibitem[ELZ]{ELZ}
H. Edelsbrunner, D. Letscher, A. Zomorodian, \textit{Topological persistence and simplification}. In Proceedings 41st Annual Symposium on Foundations of Computer Science, (2000), 454–463.

\bibitem[EKM]{EKM}
A. Elmendorf, I. Kriz, J.P. May, \textit{Modern foundations for stable homotopy theory}, Ioan Mackenzie James (ed.), Handbook of Algebraic Topology, (1995), pp. 213–253,

 \bibitem[GZ]{GZ}
 P. Gabriel, M. Zisman, \textit{Calculus of fractions and homotopy theory}. Ergebnisse der Mathematik und ihrer
Grenzgebiete, Band 35, Springer, Berlin (1967) 

\bibitem[He]{He}
A. Heller, \textit{Stable homotopy categories}
Bull. Amer. Math. Soc. 74 (1968), 28-63
\bibitem[HL]{HL}
O. Hacquard, V. Lebovici, \textit{Euler Characteristic Tools For Topological Data Analysis}. Journal of Machine Learning Research, (2024), 25. ffhal-04143938v2f

\bibitem[Ke]{Ke}
G. Kelly, \textit{Basic concepts of enriched category theory}. Reprints in Theory and Applications of Categories, No. 10,  (2005) 




\bibitem[Lu]{Lu}
J. Lurie, \textit{Higher Topos Theory}. Annals of Mathematics Studies 170, Princeton Univ. Press (2009) 

\bibitem[Ma1]{Ma1}
J.P. May, \textit{A Concise Course in Algebraic Topology}. University of Chicago Press  (1999)

\bibitem[Ma2]{Ma2}
J.P. May,
\textit{The additivity of traces in triangulated categories} Advances in Mathematics 163, (2001),  34–73 


\bibitem[Mo]{Mo}
C. Moggach, \textit{A Homotopical Categorification of the
Euler Calculus} PhD Thesis (2020)


\bibitem[Ne]{Ne}
A. Neeman,  \textit{Triangulated categories}. Annals of Mathematics Studies, Princeton University Press  (2001)

 
 \bibitem[PRSZ]{PRSZ}
 L. Polterovich, D. Rosen, K. Samvelyan, J. Zhang  \textit{Topological Persistence in Geometry and Analysis}. 	University Lecture Series, 74. American Mathematical Society, (2020)

\bibitem[PSS]{PSS}
 L. Polterovich, E. Shelukhin, and V. Stojisavljević. \textit{Persistence modules with operators in Morse and
Floer theory}. Mosc. Math. J., 17(4):757–786, (2017).

 
 

 
 \bibitem[Schw]{Schw}
S. Schwede, \textit{Stable homotopical algebra and $\Gamma$-Spaces} Mathematical Proceedings of the Cambridge Philosophical Society , Volume 126 , Issue 2, (1999), 329 - 356



\bibitem[SW]{SW}
E. Spanier, G. Whitehead, \textit{A first approximation to homotopy theory} Proc. Nat. Acad. Sci. U.S.A., 39 (1953), 655-660.

 \bibitem[Ta]{Ta}
G. Tabuada, 
\textit{Theorie homotopique des DG-categories}. (2007)
https://doi.org/10.48550/arXiv.0710.4303




 \bibitem[Ve]{Ve}
J-L. Verdier,  \textit{Des catégories dérivées des catégories abéliennes}. Astérisque, no. 239, (1996) 269  

 \bibitem[We]{We}
C. Weibel, \textit{An Introduction to Homological Algebra}. Cambridge University Press (1994) 









\bibitem[ZC]{ZC}
A. Zomorodian, G. Carlsson, \textit{Computing persistent homology}. Discrete Comput. Geom.,
33(2):  (2005) , 249–274  












 
\end{thebibliography}
\end{document}